\documentclass[11pt]{article}

\usepackage{amsmath}
\usepackage{amssymb}
\newcommand{\s}{\sum}
\newcommand{\R}{{\mathbb R}}
\newcommand{\Rp}{{\mathbb R^+}}
\newcommand{\Md}{{\mathbb M^d}}
\newcommand{\D}{{\mathbb D}(\Rp,\R)}

\newcommand{\DDDD}{{\mathbb D}(\Rp,{\mathbb R}^{4})}

\newcommand{\Dd}{{\mathbb D}(\Rp,{\mathbb R}^{d})}
\newcommand{\Ddd}{{\mathbb D}(\Rp,{\mathbb R}^{2d})}

\newcommand{\Dddddtt}{{\mathbb D}(\Rp,{\mathbb R}^{4d+2})}
\newcommand{\DMd}{{\mathbb D}(\Rp,{\mathbb M}^{d})}
\newcommand{\Dddd}{{\mathbb D}(\Rp,{\mathbb R}^{3d})}
\newcommand{\Ddddd}{{\mathbb D}(\Rp,{\mathbb R}^{4d})}
\newcommand{\Dddddd}{{\mathbb D}(\Rp,{\mathbb R}^{5d})}

\newcommand{\lra}{\longrightarrow}

\newcommand{\BH}{B^H}
\newcommand{\ZH}{Z^{H}}
\newcommand{\BB}{Z^{H}}
\newcommand{\LH}[1]{\mathbb{L}^{1/H}_{#1}}

\newcommand{\N}{{\mathbb N}}
\newcommand{\No}{{\mathbb N}\cup\{0\}}

\newcommand{\arrowk}{\mathop{\longrightarrow}_{\cal K}}

\newcommand{\RR}{{\mathbb R}}

\newcommand{\Rd}{{{{\mathbb R}^d}}}

\newcommand{\RRp}{{{{\mathbb R}^+}}}

\newtheorem{theorem}{\bf Theorem}[subsection]
\newtheorem{lemma}[theorem]{\bf Lemma}
\newtheorem{definition}[theorem]{\bf Definition}
\newtheorem{corollary}[theorem]{\bf Corollary}
\newtheorem{remark}[theorem]{\bf Remark}
\newenvironment{proof}{{\sc Proof.}}{\hfill $\Box$}

\newcommand{\nsubsection}{\setcounter{equation}{0}\subsection}

\begin{document}
\title{\bf Sweeping processes with stochastic perturbations generated
by a fractional Brownian motion}
\author{Adrian Falkowski and Leszek S\l omi\'nski\footnote{Corresponding
author. E-mail address: leszeks@mat.umk.pl;
Tel.: +48-566112954; fax: +48-566112987.}\\
 \small Faculty of Mathematics and Computer Science,
Nicolaus Copernicus University,\\
\small ul. Chopina 12/18, 87-100 Toru\'n, Poland}
\date{}
\maketitle
\begin{abstract}
We study  well-posedness  of sweeping processes with  stochastic
perturbations generated by a fractional Brownian motion and  convergence of associated numerical schemes. To this end, we first prove  new existence, uniqueness and approximation results for
deterministic sweeping processes with bounded $p$-variation and next we   apply them to  the  stochastic case.
\end{abstract}
\nsubsection{{ { Introduction}}} In the present paper we  study
well-posedness of some variants of the so-called sweeping process
introduced by Moreau in the early  70s with motivation  in
plasticity theory. In his original formulation the sweeping
process coincides with  a first order differential inclusion of
the form
\begin{equation}
\left\{\begin{array}{l}
\frac{dx}{dt}(t)\in N(C_t;x(t)), \\
x(0)=x_0\in C_0, \label{eq1.1}\\
x(t)\in C_t,
\end{array}
\right.
\end{equation}
where $C_t$  is a given convex moving set and  $N(C_t;x(t))$ is
the  inward normal cone to  $C_t$ at point $x(t)$ (see
\cite{mo1,mo2,mo3}). Many attempts have been made to generalize
Moreau's results to larger class of moving sets or more general
than (\ref{eq1.1}) differential inclusions containing
deterministic or stochastic perturbations. For instance, sweeping
by prox-regular moving sets instead of convex sets was considered
by Colombo and Goncharov \cite{cg}, Benabdellah \cite{be},
Thibault \cite{th}, Colombo and Monteiro Marques \cite{cmm}. The
study of sweeping processes with perturbations  was introduced by
Castaing, D\'uc, Ha and Valadier \cite{cdhv} and Castaing amd
Monteiro Marques \cite{cm}. The interest in the theory of sweeping
processes comes from the fact that it has numerous practical
applications in nonsmooth mechanics, analysis of hysteresis
phenomena, mathematical economics and in the modeling of switched
electrical circuits (see, e.g., the monographs by Acary, Bonnefon
and Brogliato \cite{abb}, Dr\'abek, Krej\u{c}i and Taka\u{c}
\cite{dkt},  Monteiro Marques \cite{mm} and the references
therein).

In our paper we study sweeping processes with stochastic
perturbations.  This problem was  considered earlier by Colombo
\cite{co1,co2} and  recently by  Bernicot and Venel \cite{bv}.  In
the last paper the authors give conditions ensuring well-posedness
of  $d$-dimensional stochastic differential inclusions of the form
\begin{equation}
\left\{\begin{array}{l}
dX_t\in f(t,X_t)dt+g(t,X_t)dB_t+N(C_t;X_t), \\
X_0=x_0\in C_0, \label{eq1.2}\\
X_t\in C_t,
\end{array}
\right.
\end{equation}
where $C_t$ is a given prox-regular moving set  and
$B=\{B_t\}_{t\in\Rp}$ is a standard  Brownian motion. To do this,
in proofs they combine the methods of deterministic sweeping
process theory with the  methods of stochastic differential
equations (SDEs) with reflecting boundary conditions. The use of
the methods of SDEs is  possible, because one can observe that
(\ref{eq1.2}) is equivalent to the SDE  with reflecting boundary
condition of the form
\begin{equation}\label{eq1.3}
 X_t = x_0 + \int_0^t f(s,X_{s})\,ds+\int_0^t
g(s,X_{s})\,dB_s+K_t, \quad t\in\mathbb R^+,
\end{equation}
where the integral with respect to $B$ is the classical stochastic
integral.

By a solution to (\ref{eq1.3}) we mean a pair $(X,K)$ consisting
of a process $X=\{X\}_{t\in\Rp}$ such that  $X_t\in C_t$ and the
process $K=\{K_t\}_{t\in\Rp}$,  called regulator term,  such that
$dK_t\in N(C_t;X_t)$ in appropriately defined sense.  Equation
(\ref{eq1.3}) was firstly investigated  by Skorokhod \cite{sk} for
$C_t=[0,\infty)$, $t\in\Rp$. Extensions of Skorokhod's results to
larger class of domains was studied for instance by  Tanaka
\cite{ta}, Lions and Sznitman \cite{ls}, Saisho \cite{sa}, Dupuis
and Ishi \cite{di}, S{\l}omi{\'n}ski \cite{s4}  and Rozkosz \cite{ro}. Equations of
the form (\ref{eq1.3}) also have many applications, for instance
in queueing systems, seismic reliability analysis and  finance
(see, e.g., \cite{as,dr,KS,ss} and the references therein).
Solutions of (\ref{eq1.3}) are often called solutions of
Skorokhod's SDEs  or of the Skorokhod problem.

In our paper a stochastic perturbation is generated not by a
standard Brownian motion but by a fractional Brownian motion (fBm)
 $\BH=\{\BH_t\}_{t\in\Rp}$  with Hurst index
$H>1/2$, i.e. by a continuous centered Gaussian process  with
covariance
\[
EB^H_{t_2}B^H_{t_1}=\frac12(t_2^{2H}+t_1^{2H}-|t_2-t_1|^{2H}),
\quad t_1,t_2\in\RRp.
\]
It is well known that $B^H$ is not a semimartingale and therefore
the classical stochastic integration theory for semimartingales
cannot be applied. However,  $B^H$ has $\lambda$-H\"older
continuous paths   for all $\lambda\in(0,H)$, which allows one to
define the pathwise  Riemann-Stieltjes integral with respectto fBm
(see, e.g., \cite{dn,dn1,ru}). The theory of  SDEs without
reflecting boundary condition driven by $B^H$  with the pathwise
Riemann-Stieltjes integral is at present quite  well-developed.
General results on existence  and uniqueness of solutions one can
find in Nualart and R\u{a}\c{s}canu \cite{nr}. The  viability
property for such equations is  considered in details in Ciotir
and R\u{a}\c{s}canu \cite{cr}.

Let $f:\Rd\rightarrow\Rd$, $g:\Rd\rightarrow\Rd\otimes\Rd$ be
measurable functions and $B^H$ be a $d$-dimensional fBm. Our main
purpose is to study $d$-dimensional SDE with reflecting boundary
condition of the form
\begin{equation}\label{eq1.4}
 X_t = X_0 + \int_0^t f(s,X_{s})\,ds+\int_0^t
g(s,X_{s})\,dB^H_s+K_t, \quad t\in\mathbb R^+,
\end{equation}
where the integral with respect to $B^H$ is the pathwise
Riemann-Stieltjes integral and
$C_t=[L_t,U_t]=\times_{i=1}^d[L^i_t,U^i_t]\subset\Rd$ is a moving
convex set (Here $L^i_t\leq U^i_t$, $t\in\Rp$, $i=1,\dots,d$). We
also study some generalizations of (\ref{eq1.4}).  Clearly,
(\ref{eq1.4}) is equivalent to  sweeping process of the form
(\ref{eq1.2}) with stochastic perturbation generated by fBm. The
integral form (\ref{eq1.4}) is however more convenient because in
general the process $K$ need not be of bounded variation and
therefore the use of the differential $dK_t$ would require
additional explanations.

In the recent paper by Ferrante and Rovira \cite{fr1}  the special
case of (\ref{eq1.4})  with $C_t=[0,\infty)^d$  was considered.
Using quite natural in the context of SDEs driven by $B^H$ methods
based on $\lambda$-H\"older norms  they gave conditions ensuring
the existence of solutions  and their uniqueness for some small
time interval. Some global uniqueness results  for (\ref{eq1.4})
with time homogenous  coefficients $f,g$ and
$C_t=\times_{i=1}^d[L^i_t,\infty)$ were proved in Falkowski and
S{\l}omi{\'n}ski \cite{fs}, where in contrast to \cite{fr1} the
$p$-variation norm  is used. In the present paper we also use
techniques using the $p$-variation norm. It is worth noting that
we do not assume the so-called ``interior ball condition", which
in our case means that there is $r>0$  such that $U^i_t-L^i_t>r$,
$t\in\Rp$. We even allow that $L^i_t=U^i_t$. Unfortunately, we are
not able to extend our methods  to  more general  moving convex
sets or  prox-regular moving sets (we think that it is  not
possible apart from  the case of functions $g$ depending purely on
time).

As a matter of fact, in the present paper we study more general
than (\ref{eq1.4}) equations in which the  driving processes  may
have jumps, that is equations of the form
\begin{equation}
\label{eq1.5} X_t = X_0 + \int_0^t f(s,X_{s-})\,dA_s+\int_0^t
g(s,X_{s-})\,dZ_s+K_t, \quad t\in\mathbb R^+.
\end{equation}
where  $A$ is a one-dimensional c\`adl\`ag process with locally
bounded variation and $Z$ is a $d$-dimensional c\`adl\`ag process
with locally bounded   $p$-variation for some $1<p<2$ (note that
$B^H$ has locally bounded $p$-variation only for $p\in(1/H,\infty)$).

The paper is organized as follows.

In Section 2 we consider the  deterministic extended Skorokhod
problem $x=y+k$ associated with a c\`adl\`ag $d$-dimensional
function $y$ (i.e. $y\in\Dd$) and time dependent  barriers
$l,u\in\Dd$ such that $l\leq u$, which means that $l_t\leq u_t$,
$t\in\Rp$ and  $l_0\leq y_0\leq u_0$. We show that for fixed $l,u$
the mapping $y\mapsto(x,k)$ is Lipschitz continuous  in the
$p$-variation norm. It is worth noting here that in \cite[Remark
3.6]{fr1} it is observed that $y\mapsto(x,k)$ is not Lipschitz
continuous in the $\lambda$-H\"older norm and that for  that
reason in \cite{fr1} the authors were not able to obtain global
uniqueness.

In  Section 3 we consider a deterministic version of
(\ref{eq1.5}). We give conditions ensuring the existence and
uniqueness of solutions. In the proof we use an analogue of the
Picard iteration method (we work in spaces equipped with the
$p$-variation norm). Our assumptions on the coefficients $f,g$ are
similar to those considered in \cite{nr}. Since our integrators
are c\`adl\`ag with bounded $p$-variation and need not be
$\lambda$-H\"older continuous, our theorem generalizes results
from \cite{nr} even in the trivial case where $C_t=\Rd$,
$t\in\Rp$.

Section 4 is devoted to the approximation of deterministic
solutions considered in Section 3. We consider two methods of
approximations. The first one is  an easy  to  implement
discrete-time method constructed in analogy with the classical
Euler scheme (it is an analogue of the so-called ``catching-up"
algorithm introduced by Moreau to prove the existence of  a
solution to (\ref{eq1.1})). We prove the convergence of the scheme
in the Skorokhod topology $J_1$ (in the case of continuous data we
obtain uniform  convergence on compact subsets of $\Rp$). The
second method uses stability of solutions of deterministic
versions of (\ref{eq1.5}) with respect to convergence of its
coefficients. More precisely, we consider family of solutions with
the coefficients $f_\epsilon,g_\epsilon$, $\epsilon>0$ instead of
$f,g$ and such that $\displaystyle{f_\epsilon\arrowk f}$,
$\displaystyle{g_\epsilon\arrowk g}$ as $\epsilon\to0$, which
means that $f_\epsilon,g_\epsilon$  tend to $f,g$ uniformly on
compact subsets of $\Rd$. We show that under some mild additional
assumptions on $f_\epsilon,g_\epsilon$ the associated  solutions
converge in the $p$-variation norm to the solution of equation
with coefficients $f,g$.

In Section  5 we apply our deterministic results  to show the the
existence and uniqueness of solutions  of SDEs of the form
(\ref{eq1.5}). To illustrate  how our results work in practice we
consider fractional SDEs (\ref{eq1.4}) and its simple
generalizations. We give conditions ensuring the existence and
uniqueness of their solutions and show how approximate them by a
simply to implement numerical scheme.

Section 6 contains the proof of Theorem \ref{thm1}.

In the sequel  we will use the  following  notation.
$\RRp=[0,\infty)$, $\Md$ is the space of $d\times d$ real matrices
$A$, with the matrix norm $\|A\|=\sup\{|Au|;u\in\Rd,|u|=1\}$,
where $|\cdot|$ denotes  the usual Euclidean norm in $\Rd$,
$B(0,N)=\{x\in\Rd; |x|\leq N\}$, $N\in\Rp$. $\Dd$ is the space of
c\`adl\`ag  mappings $x:\RRp\to\RR^d$, i.e. mappings which are
right continuous and admit left-hands limits equipped with the
Skorokhod topology  $J_1$. For $x\in\Dd$, $t>0$, we write
$x_{t-}=\lim_{s\uparrow t}x_s$, $\Delta x_t=x_t-x_{t-}$ and
$v_{p}(x)_{[a,b]} = \sup_\pi \sum_{i=1}^n |x_{t_i}-x_{t_{i-1}}|^p
<+\infty$, where the supremum is taken over all subdivisions
$\pi=\{a=t_0<\ldots<t_n=b\}$ of $[a,b]$. $\bar
V_p(x)_{[a,b]}=V_p(x)_{[a,b]}+|x_a|$, where
$V_p(x)_{[a,b]}=(v_p(x)_{[a,b]})^{1/p}$,  is the usual
$p$-variation norm. Moreover, for simplicity of notation we write
$v_p(x)_T=v_p(x)_{[0,T]}$, $V_p(x)_T= V_p(x)_{[0,T]}$ and $\bar
V_p(x)_T= \bar V_p(x)_{[0,T]}$. If $x\in\DMd$ then in the
definition of $p$-variation $v_p$ we use the matrix norm
$\|\cdot\|$ in place of the Euclidean norm. We write $x\leq x'$,
$x,x'\in\Dd$ if $x^i_t\leq x'^i_t$, $t\in\Rp$, $i=1,\dots,d$.

\nsubsection{Main estimates}

We begin with recalling the  definition of the extended Skorokhod
problem with time dependent reflecting barriers introduced in
\cite{bkr}. Let $y,l,u\in\Dd$ be such that $l\leq u$  and
$l_0\leq y_0\leq u_0$. We say that a~pair $(x,k)\in\Ddd$ is a
solution of the extended Skorokhod problem associated with $y$ and
barriers $l,u$ (and we write $(x, k)=ESP(y,l,u)$)  if
\begin{description}
\item[(i)]$ x_t = y_t+k_t\in[l_t,u_t]$, $t\in\Rp$,
\item[(ii)]$k_0=0$, $k=(k^1,\dots,k^d)$, where  for every
$0\leq t\leq q$  and $i=1,\dots,d$,
\begin{eqnarray*}
k^i_q-k^i_t\geq0,&&\quad\mbox{\rm if}\,\,x^i_s<u^i_s\,\,\mbox{\rm for
all}\,\,s\in(t,q],\\
k^i_q-k^i_t\leq0,&&\quad\mbox{\rm if}\,\,x^i_s>l^i_s\,\,\mbox{\rm for
all}\,\,s\in(t,q],\end{eqnarray*}and for every $t\in\Rp$, $ \Delta
k^i_t\geq0$ if $x^i_t<u^i_t$ and $\Delta k^i_t\leq0$ if $x^i_t>l^i_t$.
\end{description}
In \cite[Theorem 2.6]{bkr} it is proved that for any $y,l,u\in\Dd$
such that $l\leq u$  and  $l_0\leq y_0\leq u_0$ there exists a
unique solution $(x, k)=ESP(y,l,u)$.

\begin{remark}\label{rem1}
{\rm
(a) It is observed in \cite{sw} that
instead of (ii)  the following system  of conditions can be
considered:
for every $0\leq t\leq q$  and $i=1,\dots,d$
such that
$\inf_{s\in[t,q]}(u^i_s-l^i_s)>0$ the function $k^i$ has bounded
variation on $[t,q]$  and
\begin{equation}\label{eq2.1}\int_{[t,q]}(x^i_s-l^i_s)\,dk^i_s\leq0\quad\mbox{\rm
and}\quad\int_{[t,q]}(x^i_s-u^i_s)\,dk^i_s\leq0\end{equation}
(Here $\int_{[t,q]}w_s\,dv_s$ denotes the integral over the closed
interval $[t,q]$, which  is equal to $w_t\Delta
v_t+\int_t^qw_s\,dv_s$, where $\int_t^qw_s\,dv_s$ denotes the
usual integral over the half open interval $(t,q]$). Simple
calculations  show that the definitions from \cite{bkr} and
\cite{sw} are equivalent.

(b)  By (\ref{eq2.1}), if $u^i_ t>l^i_t$ then $(x^i_t-l^i_t)\Delta
k^i_t\leq 0$   and $(x^i_t-u^i_t)\Delta k^i_t\leq 0$.
Consequently, if $\Delta k^i_t>0$ then $x^i_t=l^i_t$ and if
$\Delta k^i_t<0$ then $x^i_t=u^i_t$, $i=1,\dots,d$. Therefore, for
every $t\in\Rp$,
\[
x_t=\max(\min((x_{t-}+\Delta y_t),u_t),l_t)\quad\mbox{\rm
and}\quad k_t=\max(\min(k_{t-},u_t-y_t),l_t-y_t),
\]
which means that $x_t$ is the  projection of $x_{t-}+\Delta y_t$
on the interval $[u_t,l_t]$  and $k_t$ is the projection of
$k_{t-}$ on the interval  $[u_t-y_t,l_t-y_t]$.

(c) In the classical Skorokhod problem it is assumed that the
function $k$ has bounded variation on  each  bounded interval
$[t,q]$, or, equivalently,  $k=k^{(+)}-k^{(-)}$,
 where $k^{(+),i}$, $k^{(-),i}$ are nondecreasing right continuous
functions with $k_0=k^{(+)}_0=k^{(-)}_0=0$ such that $k^{(+),i}$
increases only on $\{t;x^i_t={l^i}_t\}$ and  $k^{(-),i}$ increases
only on $\{t;x^i_t={u^i}_t\}$, $i=1,\dots,d$. If
$(x,k)=ESP(y,l,u)$   and $\inf_{t\leq T}(u_t-l_t)>\epsilon_T>0$,
$T\in\Rp$ then $k$ is a function of bounded variation and $(x,k)$
is a solution of the classical  Skorokhod problem (see, e.g.,
\cite[Corollary 2.4]{bkr}). }
\end{remark}

The Lipschitz continuity  of the mapping $(y,l,u)\mapsto(x,k)$ in
the supremum norm is well known. Let $(x,k)=ESP(y,l,u)$,
$(x',y')=ESP(k',l',u')$. By \cite[Theorem 2.1]{sw},
\[
\sup_{t\leq T} |x_t-x'_t|\leq2\sup_{t\leq T}|y_t-y'_t|
+\sup_{t\leq T}\max(|l_t-l'_t|,|u_t-u'_t|)
\]
and
\[
 \sup_{t\leq T}|k_t-k'_t|\leq\sup_{t\leq T}|y_t-y'_t|
+\sup_{t\leq T}\max(|l_t-l'_t|,|u_t-u'_t|).
\]
From this one can deduce the following stability result for
solutions of the extended Skorokhod  problem in the topology
$J_1$. Assume that $(x^n, k^n)=ESP(y^n,l^n,u^n)$, $n\in\N$, $(x,
k)=ESP(y,l,u)$   and $(y^n,l^n,u^n)\lra(y,l,u)$ in $\Dddd$. Then
\begin{equation}\label{eq2.2}
(x^n,k^n,y^n,l^n,u^n)\lra(x,k,y,l,u)\,\,\mbox{\rm in}\,\, \Dddddd
\end{equation}
(see \cite[Theorem 2.6]{bkr}  or \cite[Theorem 2.8]{sw}). Below we
show that in the case of fixed barriers $l,u$ the Lipschitz
continuity  of the mapping $y\mapsto(x,k)$ also holds in the
$p$-variation norm. We  first consider the case $d=1$.
\begin{theorem}\label{thm1}
Assume  that $y^1,y^2,l,u\in\D$ are such that $l_0\leq
y^1_0,y^2_0\leq u_0$ and $l\leq u$. Let $(x^j, k^j)=ESP(y^j,l,u)$,
$j=1,2$.  Then for any $T\in\Rp$
\[
\bar V_p(k^1-k^2)_T\leq \bar V_p(y^1-y^2)_T.
\]
\end{theorem}
Since our proof involves  some technical one-dimensional arguments
not associated with the rest of the paper, we defer the proof of
Theorem \ref{thm1} to Section 6.

\begin{remark}\label{rem2}
{\rm (a) The case $p=1$ was studied earlier in \cite[Theorem 2.14]{sw} (see
also \cite{sa}).

(b) In Ferrante and Rovira \cite[Remark 3.6]{fr1} it is observed
that  property stated in Theorem \ref{thm1} does not hold in
$\lambda$-H\"older norm.

(c) \cite[Example 2.15]{sw} shows that it is not possible to omit
the assumption that $l=l'$  and $u=u'$.}
\end{remark}

\begin{corollary}\label{cor1}
Assume  $y,y',l,u\in\Dd$ are such that $l_0\leq y_0,y'_0\leq u_0$.
Let $(x, k)=ESP(y,l,u)$ and $(x', k')=ESP(y',l,u)$.
Then for any $T\in\Rp$,
\[
\bar V_p(x-x')_T \leq (d+1) \bar V_p( y-y')_T\quad and\quad \bar
V_p(k-k')_T \leq d\bar V_p(y-y')_T.
 \]
\end{corollary}
\begin{proof}
By Theorem \ref{thm1},
\begin{align*}
\bar V_p(k-k')_T\leq d^{(p-1)/p}(\sum_{i=1}^dv_p(k^i-
k'^i)_{T})^{1/p}
&\leq d^{(p-1)/p}(\sum_{i=1}^d\bar V_p(  y^i-y'^i)^p_{T})^{1/p}\\
&\leq d\max_i \bar V_p( y^i- y'^i)_{T} \leq d\bar V_p(y-
y')_T.
\end{align*}
Since $\bar V_p(x-x')_T \leq \bar V_p(y-y')_T+\bar  V_p( k-k')_T$,
the proof is complete.
\end{proof}
\begin{corollary}
\label{cor2}Assume  $y,l,h,u\in\Dd$ are such that $l_0\leq y_0\leq
u_0$,  $l\leq h\leq u$. Let $(x, k)=ESP(y,l,u)$. Then for any
$T\in\Rp$,
\[
\bar V_p(x)_T \leq (d+1) \bar V_p( y)_T+d\bar V_p( h)_T\quad and
\quad\bar V_p(k)_T \leq d \bar V_p(y)_T+d\bar V_p(h)_T.
\]
\end{corollary}
\begin{proof}
Note that $(h,0)=ESP(h,l,u)$. By Corollary \ref{cor1},
\[
\bar V_p(k)_T \leq  d
\bar V_p(y-h)_T\leq d\bar V_p(y)_T+d \bar V_p(h)_T,
\]
i.e. the second inequality of the corollary is satisfied. From the
second inequality we immediately get the first one.
\end{proof}

\nsubsection{Deterministic equations with reflecting boundary condition}

Let $a\in\D$, $z,l,u\in\Dd$ be such that $V_1(a)_T$,
$V_p(z)_T<\infty$ for $T\in\Rp$ and  $l\leq
u$. We also assume that there is $h\in\Dd$ such that $l\leq h\leq
u$ and  $V_p(h)_T<\infty$ for $T\in\Rp$. This additional
assumption is indispensable to  ensure that $(x,k)=ESP(y,l,u)$
have  bounded $p$-variation  for any  bounded  $p$-variation
function $y$ (it is automatically satisfied if $\inf_{t\leq
T}(u_t-l_t)>\epsilon_T>0$, $T\in\Rp$, because  in this case $k$ is
a function of  bounded  variation).

We  consider equations with reflecting time-dependent barriers of
the form
\begin{equation}\label{eq3.1}
x_t=x_0+\int_0^t f(s,x_{s-})\,da_s+\int_0^tg(s,x_{s-})\,dz_s
+k_t,\quad t\in\Rp,
\end{equation}
where $f:\Rp\times\Rd\rightarrow\Rd$,
$g:\Rp\times\Rd\rightarrow\Md$ are given functions, the
integral with respect to $z$ is a Riemann-Stieltjes integral  and  $l_0\leq x_0\leq u_0$.  We
recall that if $w\in\DMd$, $z\in\Dd$  are such that
$V_q(w)_T<+\infty$, $V_p(z)_T<+\infty$, $T\in\Rp$, where
$1/p+1/q>1$, $p,q\geq1$, then the Riemann-Stieltjes integral
$\int_0^\cdot w_{s-} dz_s$ is  well defined (see, e.g.,
\cite{du}). Moreover, it is well known
 that for any $a<b$,
\begin{equation}\label{eq3.2}
V_p(\int_a^\cdot w_{s-} dz_s)_{[a,b]}\leq C_{p,q}\bar
V_q(w)_{[a,b)}V_p(z)_{[a,b]},
\end{equation}
where $C_{p,q}=\zeta(p^{-1}+q^{-1})$ and $\zeta$ denotes the
Riemann zeta function, i.e. $\zeta(x)=\sum_{n=1}^\infty 1/n^x$.

\begin{definition}\label{def}
{\rm  We say that a~pair $(x,k)\in\Ddd$ is a solution of
(\ref{eq3.1}) if  $V_p(x)_T<\infty$ for $T\in\Rp$  and
$(x,k)=ESP(y,l,u)$, where
\[
y_t=x_0+\int_0^t f(s,x_{s-})\,da_s+\int_0^tg(s,x_{s-})\,dz_s,\quad t\in\Rp.
\]
}
\end{definition}

We will need the following conditions  on $f,g$.
\begin{enumerate}
\item[({F})] (a) There exists $L>0$ such that
\[
|f(t,x)|\leq L(1+|x|),\quad x\in \Rd, \,t\in\Rp.
\]
(b) For every
$N\in\Rp$ there exists $L_N>0$ such that
\[|f(t,x)-f(t,y)|\leq L_N|x-y|,\quad x,y\in B(0,N), \,t\in\Rp.\]
\item[({G})] (a) There exist $\beta\in(1-1/p,1]$ and  $C^\beta>0$  such that
\[
|g(t,x)-g(s,y)| \leq C^\beta(|t-s|^\beta+|x-y|),\quad x,y\in\Rd,\,t,s\in\Rp
\]
(b) $g$  is differentiable in $x$ and for every $N\in\Rp$ there
exist $\alpha_N\in(p-1,1]$ and $C_N>0$  such that
\[
|\nabla_xg(t,x)-\nabla_xg(s,y)| \leq C_N (|t-s|^\beta+|x-y|^{\alpha_N}),
\quad x,y\in B(0,N),\,t,s\in\Rp,
\]
where $\nabla_xg(t,x)=(\nabla_xg^{i,j}(t,x))_{i,j=1,\dots,d}$
and $\displaystyle{|\nabla_xg(t,x)|^2=\sum_{k=1}^d\sum_{i=1}^d
\sum_{j=1}^d|\frac{\partial g^{i,j}}{\partial x_k}(t,x)|^2}$.
\end{enumerate}

Similar sets of conditions were considered in papers on equations
without reflecting boundary condition driven by functions
(processes) with bounded $p$-variation (see
\cite{du,k1,k2,ly,nr,ru}).

\begin{remark}{\rm
Note that under ({G})(a) for every $T\in\Rp$  there exists
${C}^{\beta,T}>0$ such that for every $t\in[0,T]$ and $x\in\Rd$,
\begin{equation}\label{eq31}
|g(t,x)|\leq {C}^{\beta,T}(1+|x|),
\end{equation}
and  for  $q= p\vee(1/\beta)$ and  every $w\in\Dd$,
\begin{equation}\label{eq32}
\bar V_{q}(g(\cdot,w))_t \leq {C}^{\beta} t^\beta+ C^\beta V_p(w)_t +|g(0,x_0)|
\leq {C}^{\beta,T}(1+\bar V_p(x)_t).
\end{equation}
Moreover, ${C}^{\beta,T}=C^{\beta}( T^\beta+1)+|g(0,0)|$.}
\end{remark}

We will approximate solutions of (\ref{eq3.1}) by using an
analogue of the Picard iteration method. Set
$(x^0,k^0)=ESP(x_0,l,u)$ and for any $n\in\N$ set
\begin{equation}
\label{eq3.3}
\left\{\begin{array}{rl}
y^n&=x_0+\int_0^\cdot f(s,x^{n-1}_{s-})\,da_s
+\int_0^\cdot g(s,x^{n-1}_{s-})\,dz_s,\\
(x^n,k^n)&=ESP(y^n,l,u),
\end{array}
\right.
\end{equation}
where the integral with respect to $z$  is the Riemann-Stieltjes
integral.  Note that (\ref{eq3.3}) is well defined if ({F})(a) and
({G})(a) are satisfied. Indeed, by Corollary \ref{cor2},
\begin{equation}\label{eq3.4}
\bar V_p(x^0)_T \leq (d+1)|x_0|+d\bar V_p(h)_T, \quad T\in\Rp
\end{equation}
and for any $n\in\N$,
\begin{align*}
\bar V_p(x^n)_T &\leq (d+1)\bar V_p(y^n)_T+d\bar V_p(h)_T\\
& \leq (d+1)\left[|x_0|+V_p(\int_0^\cdot
f(s,x_{s-}^{n-1})\,da_s)_T+V_p(\int_0^\cdot
g(s,x_{s-}^{n-1})\,dz_s)_T\right]+d\bar V_p(h)_T.
\end{align*}
Moreover,
\[
V_p(\int_0^\cdot f(s,x^{n-1}_{s-})\,da_s)_T
\leq \sup_{s\leq T}|f(s,x^{n-1}_{s-})|V_1(a)_T
\leq L(1+\bar V_p(x^{n-1})_T)V_1(a)_T
\]
and by (\ref{eq3.2})  and   (\ref{eq32}), for $q= p\vee(1/\beta)$
we have
\[
V_p(\int_0^\cdot g(s,x^{n-1}_{s-})\,dz_s)_T\leq C_{p,q}\bar
V_{q}(g(\cdot,x^{n-1}))_TV_p(z)_T \leq C_{p,q}{C}^{\beta,T}(1+\bar
V_{p}(x^{n-1})_T)V_p(z)_T.
\]
Hence, in particular,  $\bar V_p(x^n)_T<\infty$ for $n\in\N$,
$T\in\Rp$. In fact, under ({F})(a)  and  ({G})(a) we have
\begin{equation}
\label{eq3.5}
\sup_n\bar  V_p(x^n)_T<\infty,\quad T\in\Rp.
\end{equation}
To check this,  fix $T\in\Rp$ and set $C_0=(d+1)|x_0|+d\bar
V_p(h)_T$, $C_1=(d+1)\max(L,C_{p,q}{C}^{\beta,T})$. Observe that
by the above estimates  for any $t\leq T$ we have $\bar
V_p(x^0)_t\leq C_0$ and
\[
\bar V_p(x^n)_t\leq C_0+C_1(1+\bar V_{p}(x^{n-1})_t)(V_1(a)_t+V_p(z)_t),
\quad n\in\N.\]
If we set $t_1=\inf\{t;C_1(V_1(a)_t+V_p(z)_t)>1/2\}\wedge T$ then
\[\bar V_p(x^n)_{t_1-}\leq C_0+\frac12+\frac12\bar V_{p}(x^{n-1})_{t_1-},
\quad n\in\N,\]
which implies that $\sup_n\bar V_p(x^n)_{t_1-}\leq 2(C_0+1/2)$. Since
\[
|\Delta x^n_{t_1}|\leq |f(t_1,x^{n-1}_{t_1-})\Delta a_{t_1}|
+|g(t_1,x^{n-1}_{t_1-})\Delta z_{t_1}|+\max(|\Delta l_{t_1}|,|\Delta u_{t_1}|),
\]
it is clear that
\begin{equation} \label{eq3.8}
\sup_n\bar V_p(x^n)_{t_1}<\infty.
\end{equation}
Set
$t_k=\inf\{t>t_{k-1};C_1(V_1(a)_{[t_{k-1},t]}+V_p(z)_{[t_{k-1},t]})>1/2\}\wedge
T$, $k\geq2$. Modifying slightly the proof of (\ref{eq3.8}) on can
show that $\sup_n \bar V_{p}(x^n)_{[t_{k-1},t_k]} <\infty$. What
is left  is to show that $m=\inf\{k;t_k=T\}$ is finite. To see
this, without loss of generality assume that $C_1\geq1$. Observe
that $1/2<C_1(V_1(a)_{[t_{k-1},t_k]}+V_p(z)_{[t_{k-1},t_k]})\leq
2\max(V_1(a)_{[t_{k-1},t_k]},V_p(z)_{[t_{k-1},t_k]})$ for each
$k$, which implies that
$(1/4)^p<V_1(a)_{[t_{k-1},t_k]}+v_p(z)_{[t_{k-1},t_k]}$, $k\in\N$.
Consequently,
\begin{equation}\label{eq3.6}
m(\frac{1}{4})^p< \sum_{k=1}^m
(V_1(a)_{[t_{k-1},t_k]}+v_p(z)_{[t_{k-1},t_k]})
\leq (V_1(a)_{T}+v_p(z)_{T})<\infty,
\end{equation}
which  completes the proof of (\ref{eq3.5}).

\begin{theorem}\label{thm2}
Assume ({F}), ({G})  and that there exists $h\in\Dd$ such that
$l\leq h\leq u$ and $V_p(h)_T<\infty$, $T\in\Rp$. Let
$\{(x^n,k^n)\}$ denote the sequence of Picard's iterations defined
by (\ref{eq3.3}). Then for every $T\in\Rp$,
\[
\bar V_p(x^n-x)_T\rightarrow 0\quad and  \quad \bar V_p(k^n-k)_T\rightarrow 0,
\]
where  $(x,k)$ is a unique solution of  (\ref{eq3.1}).
\end{theorem}
\begin{proof} {\em Step 1.} Convergence of Picard's iteration.
Fix $T\in\Rp$.  Since $x^n_0=x^{n-1}_0=x_0$, applying  Corollary
\ref{cor1} we get
\begin{align*}
\bar V_p(x^n&-x^{n-1})_t= V_p(x^n-x^{n-1})_t\\
&\leq (d+1) V_p(\int_0^{\cdot}
f(s,x^{n-1}_{s-})-f(s,x^{n-2}_{s-})\,da_s
+\int_0^{\cdot}g(s,x^{n-1}_{s-})-g(s,x^{n-2}_{s-})\,dz_s )_t
\\
&\leq (d+1) V_p(\int_0^{\cdot}
f(s,x^{n-1}_{s-})-f(s,x^{n-2}_{s-})\,da_s)_t
\\
&\quad+(d+1)
V_p(\int_0^{\cdot}g(s,x^{n-1}_{s-})-g(s,x^{n-2}_{s-})\,dz_s )_t
\end{align*}
for $t\in[0,T]$. By (\ref{eq3.5}), $\sup_{t\leq T}|x^n_t|\leq N$
for  $n\in\N$, where $N=\sup_n\bar  V_p(x^n)_T$.  Therefore
\begin{align*}
 V_p(\int_0^{\cdot} f(s,x^{n-1}_{s-})-f(s,x^{n-2}_{s-})\,da_s)_t&\leq
L_NV_1(a)_t\sup_{s\leq t}|x^{n-1}_s-x^{n-2}_s|\\
& \leq L_NV_1(a)_t
V_p(x^{n-1}-x^{n-2})_t
\end{align*}
and by (\ref{eq3.2}),
\[
V_p(\int_0^{\cdot}g(s,x^{n-1}_{s-})-g(s,x^{n-2}_{s-})\,dz_s )_t
\leq C_{p,r} V_{r}(g(\cdot,x^{n-1})-g(\cdot,x^{n-2}))_tV_p(z)_t,
\]
where $r=(p/\alpha_N)\vee(1/\beta)$. To estimate the right
hand-side of the last inequality we  will use the following lemma.
\begin{lemma}
\label{lem3.2} If $g:\Rp\times\Rd\rightarrow\R$ satisfies {\rm
({G})} then for any  $x,y\in\Dd$, $T,N\in\Rp$ such that
$V_p(x)_T<\infty$, $V_p(y)_T<\infty$,  $\sup_{t\leq T}|x_t|$,
$\sup_{t\leq T}|y_t|\leq N$ and $r=(p/\alpha_N)\vee(1/\beta)$  we
have
\[
V_{r}(g(\cdot,x)-g(\cdot,y))_T\leq C^\beta V_{r}(x-y)_T
+C_N\sup_{t\leq T}|x_t-y_t|
\Big(T^\beta+V_p(x)_T^{\alpha_N}+V_p(y)_T^{\alpha_N}\Big).
\]
\end{lemma}
{\sc Proof of the lemma.} For every $t,s\in[0,T]$,
\begin{align*}
&|g(t,x_t)-g(t,y_t)-g(s,x_s)+g(s,y_s)|\\& \qquad=\Big|\int_0^1
\nabla_xg(t,\theta x_t+(1-\theta)y_t)(x_t-y_t)-\nabla_xg(s,\theta
x_s+(1-\theta)y_s)(x_s-y_s)d\theta\Big|.
\end{align*}
Hence
\begin{align*}
&|g(t,x_t)-g(t,y_t)-g(s,x_s)+g(s,y_s)|\\
&\qquad\leq\int_0^1 |\nabla_xg(t,\theta x_t+(1-\theta)y_t)|
|(x_t-y_t-x_s+y_s)|d\theta\\
&\qquad\quad+\int_0^1 |\nabla_xg(t,\theta x_t
+(1-\theta)y_t)-\nabla_xg(s,\theta x_s+(1-\theta)y_s)||(x_s-y_s)|d\theta\\
&\qquad\leq C^\beta|x_t-y_t-x_s+y_s|
+C_N|x_s-y_s|(|t-s|^\beta+|x_t-x_s|^{\alpha_N}+|y_t-y_s|^{\alpha_N}).
\end{align*}
Applying this estimate to each pair $t=t_i$, $s=t_{i-1}$ from an
arbitrary partition $0=t_0<t_1<\dots<t_n=T$ of $[0,T]$ and using
Minkowski's inequality, we obtain the desired result. \hfill
{$\Box$}

By Lemma \ref{lem3.2}, for  $i,j=1,\dots,d$ we have
\begin{align*}
V_{r}(g_{i,j}(\cdot,x^{n-1})&-g_{i,j}(\cdot,x^{n-2}))_t
\leq C^\beta V_{r}(x^{n-1}-x^{n-2})_t\\
&\quad+C_N\sup_{s\leq t}|x^{n-1}_s-x^{n-2}_s|
\Big(t^\beta+V_p(x^{n-1})_t^{\alpha_N}+V_p(x^{n-2})_t^{\alpha_N}\Big),
\end{align*}
which implies that
\begin{align*}
V_{r}(g(\cdot,x^{n-1})&-g(\cdot,x^{n-2}))_t \leq\sum_{i,j=1}^{d}
V_{r}(g_{i,j}(\cdot,x^{n-1})-g_{i,j}(\cdot,x^{n-2}))_t\\
&\leq (C^\beta)^{d^2} V_{r}(x^{n-1}-x^{n-2})_t\\
&\quad+(C_N)^{d^2}\sup_{s\leq t}|x^{n-1}_s-x^{n-2}_s|
\Big(t^\beta+V_p(x^{n-1})_t^{\alpha_N}+V_p(x^{n-2})_t^{\alpha_N}\Big).
\end{align*}
From the above estimates, (\ref{eq3.5}) and fact that
\[\sup_{s\leq t}|x^{n-1}_s-x^{n-2}_s|\leq V_p(x^{n-1}-x^{n-2})_t,\quad t\in\Rp
\]
we conclude that there exists   $D>0$ depending only on $C_{p,r}$,
$C^\beta$, $C_N$, $\alpha_N$, $\beta$, $L_N$, $T$ and  $d$  such
that for every $n\in\N$,
\begin{equation}
\bar V_p(x^n-x^{n-1})_t\leq D(V_1(a)_t +V_p(z)_t)\bar V_p(x^{n-1}-x^{n-2})_t.
\label{eq3.7}
\end{equation}
Set $t_1=\inf\{t>0;(D(V_1(a)_t+V_p(z)_t))\geq \frac 12\}\wedge T$
and observe that by induction,
\[
\bar V_p(x^n-x^{n-1})_{t_1-}\leq 2^{-(n-1)}\bar
V_p(x^1-x^{0})_{t_1-}, \quad n\in\N.
\]
Thus $\{x^n\}$  is  a Cauchy sequence in the space of c\`adl\`ag
functions on $[0,t_1)$ with the $p$-variation norm. Therefore
there is a c\`adl\`ag function $x$ such that $\bar
V_p(x^n-x)_{t_1-}\lra0$. This implies that $V_p(\int_0^\cdot
(f(s,x^n_{s-})-f(s,x_{s-}))\,da_s)_{t_1-}\lra0$  and
$V_p(\int_0^\cdot (g(s,x^n_{s-})-g(s,x_{s-}))\,dz_s)_{t_1-}\lra0$,
and hence that there exists a c\`adl\`ag  function $k$ such that
$\bar V_p(k^n-k)_{t_1-}\lra0$   and $(x,k)$ is a solution of
(\ref{eq3.1}) on the interval $[0,t_1)$. If we set
\[
x_{t_1}=\max(\min(x_{t_1-}+f(t_1,x_{t_1-})\Delta a_{t_1}
+g(t_1,x_{t_1-}) \Delta z_{t_1},u_{t_1}), l_{t_1})
\]
and $k_{t_1}=k_{t_1-} + \Delta x_{t_1} - (f(t_1,x_{t_1-})\Delta
a_{t_1}+g(t_1,x_{t_1-})\Delta z_{t_1})$ then  by Remark
\ref{rem1}(b), $(x,k)$ is a solution of (\ref{eq3.1}) on the
closed interval $[0,t_1]$. Moreover,
\begin{align*}
x^n_{t_1}=&\max(\min(x^n_{t_1-}+f(t_1,x^{n-1}_{t_1-})
\Delta a_{t_1}+g(t_1,x^{n-1}_{t_1-}) \Delta z_{t_1},u_{t_1}), l_{t_1})\\
&\lra\max(\min(x_{t_1-}+f(t_1,x_{t_1-})\Delta
a_{t_1}+g(t_1,x_{t_1-}) \Delta z_{t_1},u_{t_1}),
l_{t_1})=x_{t_1},
\end{align*}
which implies that $\bar V_p(x^n-x)_{t_1}\lra0$ and  $\bar
V_p(k^n-k)_{t_1}\lra0$. It is easy to see that we can apply the
arguments used above to the interval $[t_1,t_2]$ with
$t_2=\inf\{t>0;D(V_1(a)_{[t_1,t)}+V_p(z)_{[t_1,t)})\geq \frac
12\}\wedge T$, and then to intervals $[t_2,t_3]$,
$[t_3,t_4],\dots$ Since $V_1(a)_T<\infty$ and $V_p(z)_T<\infty$,
in finitely many steps we are able to construct the solution
$(x,k)$ of (\ref{eq3.1}) on the whole interval $[0,T]$ and to show
that $\bar V_p(x^n-x)_{T}\lra0$ and $\bar V_p(k^n-k)_{T}\lra0$.

{\em Step 2.} Uniqueness of solutions of (\ref{eq3.1}). Assume
that there exists two solutions $(x^1,k^1)$  and $(x^2,k^2)$. Let
$t_1$ be defined as in Step 1. Using arguments from Step 1  we
show that $V_p(x^1-x^2)_{t_1-}\leq\frac12\bar
V_p(x^1-x^2)_{t_1-}$\,, which implies that $x^1=x^2$ on $[0,t_1)$.
Since by Remark \ref{rem1}(b) we know that
\[
x^j_{t_1}=\max(\min(x^j_{t_1-}
+f(t_1,x^j_{t_1-})\Delta a_{t_1}+g(t_1,x^j_{t_1-}) \Delta
z_{t_1},u_{t_1}),l_{t_1}),\quad j=1,2,
\]
it is clear that $x^1=x^2$  on the  closed interval $[0,t_1]$.
Applying the above argument to intervals $[t_1,t_{2}]$,
$[t_2,t_3]$, $\dots$ we show in finitely many steps that $x^1=x^2$
on $[0,T]$ for every $T\in\Rp$.
\end{proof}

\nsubsection{Discrete-time approximation and stability of solutions}

We assume that $l,h,u\in\Dd$ are such that  $l\leq
h\leq u$, $V_p(h)_T<\infty$ for $T\in\Rp$.  Let $x_0\in[l_0,u_0]$  and $a\in\D$,
$z\in\Dd$ be such that $V_1(a)_T<\infty$ and $V_p(z)_T<\infty$,
$T\in\Rp$.

Set $x^n_0=x_0$, $k^n_0=0$ and
\begin{equation}
\left\{\begin{array}{ll}
\Delta y^n_{(k+1)/n}&=f(k/n,x^n_{k/n})(a_{(k+1)/n}-a_{k/n})
+g(k/n,x^n_{k/n})(z_{(k+1)/n}-z_{k/n}),\smallskip \\
x^n_{(k+1)/n}&=\max\big(\min(x^n_{k/n}
+\Delta y^n_{(k+1)/n},u_{(k+1)/n}),l_{(k+1)/n}\big), \label{eq4.1}
\smallskip \\
k^n_{(k+1)/n}&=k^n_k+(x^n_{(k+1)/n}-x^n_{k/n})-\Delta
y^n_{(k+1)/n}
\end{array}
\right.
\end{equation}
and $x^n_t=x^n_{k/n}$, $k^n_t=k^n_{k/n}$, $t\in[k/n,(k+1)/n)$,
$k\in\N\cup\{0\}$. Since
\[
x^n_{(k+1)/n}=\Pi_{C_{{(k+1)/n}}}\big(x^n_{k/n} +\Delta
y^n_{(k+1)/n}\big),\quad k\in\N\cup\{0\},
\]
where $\Pi_{C_{{(k+1)/n}}}$ denotes the  projection on the set
$C_{{(k+1)/n}}=[l_{(k+1)/n},u_{(k+1)/n}]$, (\ref{eq4.1}) is the
well known Euler scheme for (\ref{eq3.1}) (see, e.g., \cite{s3}).
It is also an analogue  of the so-called ``catching-up" algorithm
introduced by Moreau to prove the existence of  a solution of
(\ref{eq1.1}) (see, e.g., \cite{aht}).

\begin{theorem}\label{thm3}
Let  $\{(x^n,k^n)\}$ be a sequence of
approximations defined by (\ref{eq4.1}).  If $f,g$ satisfy
\mbox{\rm ({F}), ({G})} and moreover $f$  is continuous then
\begin{equation}
\label{eq3.16} (x^n,k^n,l^n,u^n)\lra (x,k,l,u)\quad\mbox{\rm
in}\,\,\Ddddd,\end{equation} where $l^n_t=l_{k/n}$,
$u^n_t=u_{k/n}$, $t\in[k/n,(k+1)/n)$, $k\in\N\cup\{0\}$ and
$(x,k)$ is a  unique solution of (\ref{eq3.1}).
\end{theorem}
\begin{proof}
Fix $T\in\Rp$ and set $b_T=V_1(a)_T+V_p(z)_T+\bar V_p(h)_T
+\sup_{t\leq T}\max(|\Delta l_t|,|\Delta u_t|)$.
First we show that if $(x,k)$ satisfies (\ref{eq3.1}) then
\begin{equation}\label{eq3.10}
\bar V_p(x)_T\leq D\quad \mbox{\rm where } D\,\, \mbox{\rm depends
only on } d, x_0, L, {C}^{\beta,T}, \beta\,\,\mbox{\rm
and}\,b_T.\end{equation} By Corollary \ref{cor2}, for any $t\leq
T$,
\begin{align*}
\bar V_{p}(x)_t&\leq (d+1)\bar V_p(y)_t+d\bar V_p(h)_t\\
& \leq (d+1)\left[|x_0|+V_p(\int_0^\cdot
f(s,x_{s-})\,da_s)_t+V_p(\int_0^\cdot
g(s,x_{s-})\,dz_s)_t\right]+d\bar V_p(h)_t.
\end{align*}
We have $V_p(\int_0^\cdot f(s,x_{s-})\,da_s)_t\leq
V_1(a)_t\,\sup_{s\leq t}|f(s,x_{s-})| \leq LV_1(a)_t\,(1+\bar
V_p(x)_t)$ and, by (\ref{eq3.2}) and (\ref{eq32}),
\begin{align*}
V_p(\int_0^\cdot g(s,x_{s-})\,dz_s)_t&\leq
C_{p,p\vee (1/\beta)}\bar V_{p\vee (1/\beta)}(g(\cdot,x))_tV_p(z)_t\\
&\leq C_{p,p\vee (1/\beta)}{C}^{\beta,T} (1+\bar V_{p}(x)_t)V_p(z)_t.
\end{align*}
Hence there is $C_0>0$ depending only on $d$, $x_0,\bar V_p(h)_T$
and $C_1>0$  depending on $d,L,\beta,{C}^{\beta,T}$  such that
\[
\bar V_{p}(x)_t \leq C_0+C_1(1+\bar V_p(x)_t)( V_1(a)_t+V_p(z)_t).
\]
Set $t_1=\inf\{t;C_1(V_1(a)_t+V_p(z)_t)>\frac12\}\wedge T$.
By the above,
\[
\bar V_{p}(x)_{[0,t_1)} \leq C_0+\frac{1}{2}
+\frac12 \bar V_p(x)_{[0,t_1)},
\]
which implies that $\bar V_{p}(x)_{[0,t_1)} \leq 2(C_0+1/2)$.
Since by (\ref{eq31}),
\begin{align*}
|\Delta x_{t_1}|&\leq |f(t_1,x_{t_1-})\Delta a_{t_1}|
+|g(t_1,x_{t_1-})\Delta z_{t_1}|+\max(|\Delta l_{t_1}|,|\Delta u_{t_1}|)\\
&\leq L_T(1+|x_{t_1-}|)|\Delta a_{t_1}|+\tilde{C}_{\beta,T}
(1+|x_{t_1-}|)|\Delta z_{t_1}|+\max(|\Delta l_{t_1}|,|\Delta u_{t_1}|),
\end{align*}
it is clear that $\bar V_{p}(x)_{[0,t_1]}\leq D$ where $D$ depends
only on $d, x_0, L, {C}^{\beta,T}, \beta$ and $b_T$. If we set
$t_k=\inf\{t>t_{k-1};C_1(V_1(a)_{[t_{k-1},t]}+V_p(z)_{[t_{k-1},t]})>\frac12\}\wedge
T$, $k=2,3,\dots$, then for by the same arguments $
V_{p}(x)_{[t_{k-1},t_k]}\leq D$ where $D$ is depending only on $d,
x_{t_{k-1}}, L, {C}^{\beta,T}, \beta$ and $b_T$. Since
$m=\inf\{k;t_k=T\}$ is bounded (similarly to (\ref{eq3.6}) one can
check that $m\leq 4^p(V_1(a)_{T}+v_p(z)_{T})$), this  completes
the proof of (\ref{eq3.10}).

Now set $a^{n}_t=a_{k/n}$, $z^{n}_t=z_{k/n}$, $h^n_t=h_{k/n}$,
$\rho^n_t=k/n$, $t\in[k/n,(k+1)/n)$, $k\in\N\cup\{0\}$. It is an
elementary check that
\[
\left\{\begin{array}{l}(x^n,k^n)=ESP(y^n,l,u),\\
\mbox{\rm where}\,\,y^n=x_0+\int_0^t f(\rho^n_{s-},x^n_{s-})\,da^n_s
+\int_0^t g(\rho^n_{s-},x^n_{s-})\,dz^n_s,\quad t\in\Rp.
\end{array}
\right.
\]
Clearly, for any $n\in\N$,
\begin{equation}\label{eq3.110}
V_1(a^n)_T\leq V_1(a)_T,\, V_p(z^n)_T\leq V_p(z)_T\,
\mbox{\rm and}\,V_p(h^n)_T\leq V_p(h)_T.
\end{equation}
Combining (\ref{eq3.10}) with (\ref{eq3.110}) we get
\begin{equation}
\label{eq3.11}
\sup_n\bar V_p(x^n)_T<\infty,\quad T\in\Rp.
\end{equation}
Let $x^{(n)}_t=x_{k/n}$, $k^{(n)}_t=k_{k/n}$, $t\in[k/n,(k+1)/n)$,
$k\in\N\cup\{0\}$, denote the discretization of the solution
$(x,k)$. By using \cite[Chapter 3, Proposition 6.5]{ek} and
\cite[Chapter VI, Proposition 2.2]{js} one can check that
$(x^{(n)},a^n,z^n,l^n,u^n,\rho^n)\lra(x,a,z,l,u,I)$ in $\Dddddtt$,
where $I_s=s$, $s\in\Rp$. From this and an easy extension of
\cite[Proposition 2.9]{jmp} to functions with bounded
$p$-variation it follows that
\begin{align*}
\Big(\bar y^n=x_0+&\int_0^\cdot f(\rho^n_{s-},x^{(n)}_{s-})\,da^n_s
+\int_0^\cdot g(\rho^n_{s-},x^{(n)}_{s-})\,dz^n_s,l^n,u^n\Big)\\
&\lra\Big(y =x_0+\int_0^\cdot f(s,x_{s-})\,da_s+\int_0^\cdot
g(s,x_{s-})\,dz_s,l,u\Big)\quad \mbox{\rm in }\Dddd.
\end{align*}
By the above and (\ref{eq2.2}),
\begin{equation}\label{eq3.12}
(\bar x^n, \bar k^n, \bar y^n, l^n,u^n)\lra (x,k,y,l,u)
\quad \mbox{\rm in }\Dddddd,
\end{equation}
where $(\bar x^n,\bar k^n)=ESP(\bar y^n,l^n,u^n)$, $n\in\N$.
Moreover, analysis similar to that in the proof of (\ref{eq3.10})
shows that
\begin{equation}\label{eq3.13}
\sup_n\bar V_p(\bar x^n)_T<\infty,\quad T\in\Rp.
\end{equation}
By (\ref{eq3.12})  and \cite[Chapter VI, Proposition 2.2]{js},
$(\bar x^n,x^{(n)})\lra (x,x)$ in $\Ddd$, which implies that
\[\sup_{t\leq T}|\bar x^n_t-x^{(n)}_t|\lra0,\quad T\in\Rp.
\]
Combining the above convergence with (\ref{eq3.13}) and  the fact
that $\bar V_p(x^{(n)})_T\leq \bar V_p(x)_T<\infty$, for every
$\epsilon>0$ we obtain
\begin{equation}
\label{eq3.14} \bar V_{p+\epsilon}(\bar x^n-x^{(n)})_T\leq{\rm
Osc}(\bar x^n-x^{(n)})_T^{1-p/(p+\epsilon)}\bar V_p(\bar
x^n-x^{(n)})_T^{p/(p+\epsilon)}\lra0,\quad T\in\Rp,
\end{equation}
where Osc$(x)_T=\sup_{s,t\leq T}|x_t-x_s|$. Fix $\epsilon>0$. By
Corollary \ref{cor1}, for  any  $n\in\N$  and $t\leq T$,
\begin{align*}
&\bar V_{p+\epsilon}(x^n-\bar x^{n})_t\\&\quad
\leq (d+1)\bar V_{p+\epsilon}(\int_0^{\cdot}
f(\rho^n_{s-},x^{n}_{s-})-f(\rho^n_{s-},x^{(n)}_{s-})\,da^n_s
+\int_0^{\cdot}g(\rho^n_{s-},x^{n}_{s-})-g(\rho^n_{s-},x^{(n)}_{s-})\,dz^n_s )_t\\
&\quad\leq (d+1)\bar V_{p+\epsilon}(\int_0^{\cdot}
f(\rho^n_{s-},\bar x^{n}_{s-})-f(\rho^n_{s-},x^{(n)}_{s-})\,da^n_s
+\int_0^{\cdot}g(\rho^n_{s-},\bar x^{n}_{s-})
-g(\rho^n_{s-},x^{(n)}_{s-})\,dz^n_s )_t\\
&\quad\quad+ (d+1)\bar V_{p+\epsilon}(\int_0^{\cdot}
f(\rho^n_{s-},x^{n}_{s-})-f(\rho^n_{s-},\bar x^{n}_{s-})\,da^n_s
+\int_0^{\cdot}g(\rho^n_{s-},x^{n}_{s-})
-g(\rho^n_{s-},\bar x^{n}_{s-})\,dz^n_s )_t\\
&\quad=I^{n,1}_t+I^{n,2}_t.
\end{align*}
From (\ref{eq3.13}), the estimates from the proof of Theorem
\ref{thm2} and the fact that $\bar V_{p+\epsilon}(v)_T\leq\bar
V_p(v)_T$ one can deduce that there is $D>0$   such that for any
$n\in\N$,
\[
I^{n,1}_T\leq D(V_1(a^n)_T +V_p(z^n)_T)\bar V_{p+\epsilon}(\bar
x^{n}-x^{(n)})_T.
\]
This together with (\ref{eq3.110}) and (\ref{eq3.14}) shows that
$\lim_{n\to\infty}I^{n,1}_T=0$. The same arguments  and
(\ref{eq3.11}) show that there is $D>0$ such that for every $t\leq
T$,
\[
I^{n,2}_t\leq  D(V_1(a^n)_t+V_p(z^n)_t)\bar V_{p+\epsilon}(x^{n}-\bar x^{n})_t.
\]
If we set $t_1=\inf\{t; D(V_1(a)_t+V_p(z)_t)>1/2\}\wedge T$ then
by (\ref{eq3.110}),
\[
\bar V_{p+\epsilon}(x^n-\bar x^{n})_{t_1-}\leq I^{n,1}_T
+\frac12\bar V_{p+\epsilon}(x^{n}-\bar x^{n})_{t_1-}\,,
\]
which implies that $\bar V_{p+\epsilon}(x^n-\bar
x^{n})_{t_1-}\to0$. This and (\ref{eq3.14}) imply that $\bar
V_{p+\epsilon}(x^n- x^{(n)})_{t_1-}\to0$. Note that
\[
x^n_{t_1}=\max(\min(x^n_{t_1-}+f(t_1,x^n_{t_1-})\Delta a^n_{t_1}+g(t_1,x^n_{t_1-})
\Delta z^n_{t_1},u^n_{t_1}), l^n_{t_1}),\quad n\in\N.
\]
If $t_1$ is a nonrational number then  $\Delta a^n_{t_1}=\Delta
z^n_{t_1}=\Delta l^n_{t_1}=\Delta u^n_{t_1}=\Delta
x^{(n)}_{t_1}=0$. Hence $x^n_{t_1}=x^n_{t_1-}$  and
$x^{(n)}_{t_1}=x^{(n)}_{t_1-}$\, which implies that
\begin{equation}\label{eq3.15}
\bar V_{p+\epsilon}(x^n- x^{(n)})_{t_1}\lra0.
\end{equation}
In case  $t_1$ rational, set $I=\{n; \mbox{\rm there is
}k\,\,\mbox{\rm   such that  }t_1=k/n\}$ and observe that if $n\in
I$ then $\Delta a^n_{t_1}=\Delta a_{t_1}$, $\Delta
z^n_{t_1}=\Delta z_{t_1}$, $\Delta l^n_{t_1}=\Delta l_{t_1}$,
$\Delta u^n_{t_1}=\Delta u_{t_1}$  and $\Delta
x^{(n)}_{t_1}=\Delta x_{t_1}$. Consequently,
$\lim_{n\to\infty,\,n\in I}x^n_{t_1}=x_{t_1}$  and
$\lim_{n\to\infty,\,n\in I}x^{(n)}_{t_1}=x_{t_1}$. On the other
hand, if $n\notin I$ then  $x^n_{t_1}=x^n_{t_1-}$  and
$x^{(n)}_{t_1}=x^{(n)}_{t_1-}$\,, which completes the proof of
(\ref{eq3.15})  in case of  rational $t_1$.    By  the same
arguments in finitely many steps we show that
\[
\bar V_{p+\epsilon}(x^n- x^{(n)})_{T}\lra0.
\]
Therefore $\sup_{t\leq T}|x^n_t-x^{(n)}_t|\to0$, from which we
deduce  that $\sup_{t\leq T}|k^n_t-k^{(n)}_t|\to0$, which together
with (\ref{eq3.12}) completes the proof of (\ref{eq3.16}).
\end{proof}

\begin{corollary}\label{cor4}
Under the assumptions of Theorem \ref{thm3}, for any $T\in\Rp$,
\begin{equation}
\label{eq3.9}
\max_{k/n\leq T}|x^n_{k/n}-x_{k/n}|\lra0\quad
and\quad\max_{k/n\leq T}|k^n_{k/n}-k_{k/n}|\lra0,
\end{equation}
where $(x,k)$ is a unique solution of (\ref{eq3.1}).
\end{corollary}
\begin{proof}
By (\ref{eq3.16})  and \cite[Chapter VI, Proposition 2.2]{js},
\[(
x^n,x^{(n)},k^n,k^{(n)})\lra (x,x,k,k)\quad\mbox{\rm
in}\,\,\Ddddd.
\]
This implies that $x^n-x^{(n)}\to0$ in $\Dd$ and $k^n-k^{(n)}\to0$
in $\Dd$,  which is equivalent to (\ref{eq3.9}).
\end{proof}

\begin{theorem}\label{thm4}
Assume \mbox{\rm ({F}), ({G})} and that there exists $h\in\Dd$
such that $l\leq h\leq u$ and $V_p(h)_T<\infty$ for $T\in\Rp$. For
$\epsilon>0$ let $l_0\leq x^\epsilon_0\leq u_0$  and let
$f_\epsilon,g_\epsilon$ be functions satisfying \mbox{\rm ({F}),
({G})} with constants $L,\beta,C^\beta$ not depending on
$\epsilon$. If $(x^\epsilon,k^\epsilon)$ denotes a solution of
(\ref{eq3.1}) with $x_0,f,g$ replaced by
$x^\epsilon_0,f_\epsilon,g_\epsilon$ and $x^\epsilon_0\to x_0$,
$\displaystyle{f_\epsilon\arrowk f}$,
$\displaystyle{g_\epsilon\arrowk g}$ then  for every $T\in\Rp$,
\[
\bar V_p(x^\epsilon-x)_T\rightarrow 0\quad and
\quad \bar V_p(k^\epsilon-k)_T\rightarrow 0,
\]
where  $(x,k)$ is a unique solution of  (\ref{eq3.1}).
\end{theorem}
\begin{proof}
First observe that by (\ref{eq3.10}),
\begin{equation}\label{eq4.7}
\sup_{\epsilon>0}\bar V_p(x^\epsilon)_T<\infty,\quad
T\in\Rp.
\end{equation}
Fix $T\in\Rp$. By Corollary \ref{cor1}, for every $t\in[0,T]$,
\begin{align*}
\bar V_p(x^\epsilon-x)_t&\leq (d+1)\Big(|x^\epsilon_0-x_0|+
V_p(\int_0^{\cdot}
f_\epsilon(s,x^{\epsilon}_{s-})-f(s,x_{s-})\,da_s)_t\\
&\quad+V_p(\int_0^{\cdot}g_\epsilon(s,x^\epsilon_{s-})-g(s,x_{s-})\,dz_s )_t\Big)\\
&\leq (d+1)\Big(|x^\epsilon_0-x_0|+ V_p(\int_0^{\cdot}
(f_\epsilon-f)(s,x^\epsilon_{s-})\,da_s)_t
+V_p(\int_0^{\cdot}(g_\epsilon-g)(s,x^\epsilon_{s-})\,dz_s )_t\\
&\quad+ V_p(\int_0^{\cdot}
f(s,x^{\epsilon}_{s-})-f(s,x_{s-})\,da_s)_t
+V_p(\int_0^{\cdot}g(s,x^\epsilon_{s-})-g(s,x_{s-})\,dz_s )_t\Big)\\
&=(d+1)\Big(|x^\epsilon_0-x_0|+I^{\epsilon,1}_t+I^{\epsilon,2}_t
+I^{\epsilon,3}_t+I^{\epsilon,4}_t\Big).
\end{align*}
Set $N=\sup_\epsilon\bar V_p(x^\epsilon)_T$, which is finite by to
(\ref{eq4.7})  Clearly, $\sup_{t\leq T}|x^\epsilon_t|\leq N$ and
\[
I^{\epsilon,1}_T= V_p(\int_0^{\cdot}
(f_\epsilon-f)(s,x^\epsilon_{s-})\,da_s)_{T}
\leq \sup_{\substack{v\in B(0,N) \\
t\in[0,T]}}|f_\epsilon(t,v)-f(t,v)|V_1(a)_T \xrightarrow[\epsilon\to0]{}0.
\]
Since $p<2$,  there exists  $\gamma\in(1-1/p,\beta\wedge(1/p))$.
Therefore  by  (\ref{eq3.2}), (\ref{eq32}) and (\ref{eq4.7})
there is $C>0$ depending only on
$\gamma,C_{p,1/\gamma},L,\beta,C^\beta,C^{\beta,T}$ and $V_p(z)_T$
such that
\begin{align*}
I^{\epsilon,2}_T&=V_p(\int_0^{\cdot}(g_\epsilon-g)(s,x^\epsilon_{s-})\,dz_s)_{T}
\leq
C_{p,1/\gamma}\bar V_{1/\gamma}((g_\epsilon-g)(\cdot,x^\epsilon))_{T}V_p(z)_T\\
&\leq C_{p,1/\gamma}({\rm Osc}
((g_\epsilon-g)(\cdot,x^\epsilon))_{T}^{(1/\gamma-p\vee(1/\beta))\gamma}
V_{p\vee(1/\beta)}((g_\epsilon-g)(\cdot,x^\epsilon))_T^{(\beta\wedge(1/p))\gamma}\\
&\quad+|(g_\epsilon-g)(0,x^\epsilon_0)|)V_p(z)_T\\
&\leq C\sup_{\substack{v\in B(0,N) \\ t\in[0,T]}}
(|g_\epsilon(t,v)-g(t,v)|^{1-(p\vee(1/\beta))\gamma}+|g_\epsilon(0,v)-g(0,v)|).
\end{align*}
Consequently,
\begin{equation}
\label{eq4.8}
I^{\epsilon,1}_T+I^{\epsilon,2}_T\xrightarrow[\epsilon\to0]{}0.
\end{equation}
Using once again  (\ref{eq4.7}) and estimates from the proof of
Theorem \ref{thm2} we check  that there is $D>0$ such that for
every $t\leq T$,
\[
(d+1)(I^{\epsilon,3}_t+I^{\epsilon,4}_t)
\leq  D(V_1(a)_t+V_p(z)_t)\bar V_{p}(x^{\epsilon}-x)_t.
\]
Similarly to the proof of  Theorem \ref{thm2} we set $t_1=\inf\{t;
D(V_1(a)_t+V_p(z)_t)>1/2\}\wedge T$. Observe that
\[
\bar V_{p}(x^\epsilon- x)_{t_1-}\leq(d+1)(|x^\epsilon_0-x_0|
+I^{\epsilon,1}_T +I^{\epsilon,2}_T)+\frac12\bar
V_{p}(x^{\epsilon}- x)_{t_1-}\,.
\]
From this and (\ref{eq4.8}) we deduce that $\bar
V_{p}(x^\epsilon-x)_{t_1-}\to0$. Using  arguments from the proof
of Theorem \ref{thm2} we show that this implies that $ \bar
V_{p}(x^\epsilon-x)_{t_1}\xrightarrow[\epsilon\to0]{}0$. Applying
this argument to (finitely many) intervals $[t_i,t_{i+1}]$ we
prove the theorem.
\end{proof}

\nsubsection{Applications to stochastic processes}

In this section we apply  our  deterministic results to SDEs with
reflecting boundary condition. Let $(\Omega, {\cal F}, ({\cal
F}_t), P)$ be a filtered probability space and let $A$ be an
$({\cal F}_t)$ adapted process with trajectories in $\D$,
$Z,L,U,H$ be $({\cal F}_t)$ adapted processes with trajectories in
$\Dd$ such that  $L\leq H\leq U$ and $P(V_1(A)_T<\infty)=1$,
$P(V_p(Z)_T<\infty)=1$, $P(V_p(H)_T<\infty)=1$ for every
$T\in\Rp$. Note that $Z$ need not be a semimartingale. However, it
is a  Dirichlet process and a $p$-semimartigale in the sense
considered in \cite{cms} and \cite{k1,k2}.

\begin{definition}\label{def1}
{\rm Let $L\leq U$ and $X_0$ be an ${\cal F}_0$ measurable  random
vector such that  $L_0\leq X_0\leq U_0$. We  say that a pair
$(X,K)$  of  $({\cal F}_t)$ adapted processes with trajectories in
$\Dd$  such that $P(V_p(X)_T<\infty)=1$ for $T\in\Rp$ is a strong
solution of (\ref{eq1.1})  if $(X,K)=ESP(Y,L,U)$, where
\[
Y_t=X_0+\int_0^tf(s,X_{s-})\,dA_s+\int_0^tg(s,X_{s-})\,dZ_s,\quad t\in\Rp.
\]}
\end{definition}

\begin{theorem} \label{thm5}
Assume \mbox{\rm({F}), ({G})}. Let $L,\,U,\,H$ be $({\cal F}_t)$
adapted processes with trajectories in $\Dd$ such that  $L\leq
H\leq U$  and  $P(V_p(H)_T<\infty)=1$. If $X_0$ is an ${\cal F}_0$
measurable random vector such that  $L_0\leq X_0\leq U_0$ then
(\ref{eq1.1}) has a  unique strong solution $(X,K)$. Moreover, if
we define $\{(X^n,K^n)\}$  to be a sequence of Picard's iterations
for (\ref{eq1.1}), i.e. $(X^0,K^0)=ESP(X_0,L,U)$ and for each
$n\in\N$, $Y^n=X_0+\int_0^\cdot f(s,X^{n-1}_s)\,dA_s+\int_0^\cdot
g(s,X^{n-1}_s)\,dZ_s$ and $(X^n,K^n)=ESP(Y^n,L,U)$ then for every
$T\in\Rp$,
\[
\bar V_p(X^n-X)_T\rightarrow0,\,\,\mbox{
P-a.s.}\quad and\quad \bar V_p(K^n-K)_T\rightarrow0,\,\,
\mbox{P-a.s.}
\]
\end{theorem}
\begin{proof}
From Theorem \ref{thm2} we deduce that for every $\omega\in\Omega$
there exists a unique solution
$(X(\omega),K(\omega))=ESP(Y(\omega),L(\omega),U(\omega))$ and for
every $T\in\Rp$,
\[
\bar V_p (X^n(\omega)-X(\omega))_T\rightarrow0,\,\,\mbox{
P-a.s.},\quad\bar V_p (K^n(\omega)-K(\omega))_T\rightarrow0,\,\,
\mbox{P-a.s.}
\]
Since for each $n\in\N$ the pair $(X^n,K^n)$  is $({\cal F}_t)$
adapted, the pair of limit processes  $(X,K)$  is $({\cal F}_t)$
adapted as well, which completes the proof.
\end{proof}
\medskip

Let $\BH$ be a fractional Brownian  motion (fBm) with Hurst index
$H>1/2$
and let $\sigma:\RRp\to\RR$ is a measurable function  such that
\[
\|\sigma\|_{\LH{[0,T]}}:=(\int_0^T|\sigma_s|^{1/H}ds)^H<\infty,\quad
T\in\Rp.
\]
One can observe that the process   $\BB=\int_0^\cdot\sigma_s\,d\BH_s$ is
 a  centered Gaussian process with continuous
trajectories.
Moreover, by \cite[Theorem 1.1]{mmv}, for every $r>0$,
\begin{align*}
E|\ZH_{t_2}-\ZH_{t_1}|^r \leq C(r,H)
\Big(\int_{t_1}^{t_2}|\sigma_{s}|^{1/H} ds\Big)^{rH}
\end{align*}
for all $0\le t_1\le t_2$. Hence for any subdivision
$\pi=\{0=t_0<\ldots<t_n=T\}$ of $[0,T]$ we have
\[
\sum_{i=1}^n(E|\BB_{t_i}-\BB_{t_{i-1}}|)^{1/H}\leq
(C(1,H))^{1/H}
\sum_{i=1}^n\Big(\int_{t_{i-1}}^{t_1}|\sigma_{s}|^{1/H}ds\Big)=
(C(1,H))^{1/H}\|\sigma\|^{1/H}_{\LH{[0,T]}}.
\]
Therefore  from \cite[Theorem 3.2]{jm} it follows that if $p>1/H$
then $P(V_p(\BB)_T<\infty)=1$ for $T\in\Rp$ (note also that $Z^H$
is a Dirichlet process from the class ${\cal D}^{1/H}$ studied in
\cite{cs}). To  approximate  $Z^H$ one can use the methods
developed in \cite{sz2}.

We now show how to apply Theorem \ref{thm5} and our approximation results of Section 4 to
fractional SDEs with reflecting boundary condition. Let
$B^H=(B^{H,1},\dots,B^{H,d})$, where  $B^{H,1},\dots,B^{H,d}$ are
independent  fractional Brownian motions, and let
$Z^H=(Z^{H,1},\dots,Z^{H,d})$, where
$Z^{H,i}=\int_0^\cdot\sigma^i_s\,d B^{H,i}_s$ with
$\sigma^i:\RRp\to\RR$ such that $\|\sigma^i\|_{\LH{[0,T]}}<\infty$
for $T>0$, $i=1,\dots,d$. Let $a:\RRp\to\RR$ be a continuous
function with locally bounded variation. We consider fractional
SDEs of the form
\begin{equation}\label{eq1.6}
 X_t = X_0 + \int_0^t f(s,X_{s})\,da_s+\int_0^t
g(s,X_{s})\,dZ^{H}_s+K_t, \quad t\in\mathbb R^+.
\end{equation}
Clearly, (\ref{eq1.6})  generalizes classical fractional SDEs
driven by $B^H$ and is a particular case of (\ref{eq1.5}).

For any $n\in\N$ we set
\begin{equation}\label{eq5.9}
X^n_t=X^n_{k/n},\quad K^n_t=K^n_{k/n},\quad
t\in[k/n,(k+1)/n),\quad k\in\N\cup\{0\},
\end{equation}
where $X^n_0=X_0$, $K^n_0=0$ and
\begin{equation}\!\!\left\{\begin{array}{ll}
\Delta Y^n_{(k+1)/n}&=f(k/n,X^n_{k/n})(a_{(k+1)/n}-a_{k/n})
+g(k/n,X^n_{k/n})(\ZH_{(k+1)/n}-\ZH_{k/n}), \smallskip\\
X^n_{(k+1)/n}&=\max\big(\min(X^n_{k/n}
+\Delta Y^n_{(k+1)/n},U_{(k+1)/n}),L_{(k+1)/n}\big),\label{eq5.10}
\smallskip\\
K^n_{(k+1)/n}&=K^n_{k/n}+(X^n_{(k+1)/n}-X^n_{k/n})-\Delta
Y^n_{(k+1)/n}.
\end{array}
\right.
\end{equation}

\begin{corollary} \label{cor10}
Assume \mbox{\rm({F}), ({G})}. Let $L,\,U,\,H$ be $({\cal F}_t)$
adapted processes with continuous trajectories such that  $L\leq
H\leq U$  and   $P(V_p(H)_T<\infty)=1$ for $T\in\Rp$. If $X_0$ is
an ${\cal F}_0$ measurable random vector such that $L_0\leq
X_0\leq U_0$   then (\ref{eq1.6}) has a  unique strong solution
$(X,K)$. Moreover, if $\{(X^n,K^n)\}$ is a sequence of
approximation defined by (\ref{eq5.9}) and (\ref{eq5.10}) then for
every $T\in\Rp$,
\begin{equation} \label{eq5.11}
\sup_{t\leq T}|X^n_t-X_t|\rightarrow0, \,\, P\mbox{-a.s.},\quad
\sup_{t\leq T}|K^n_t-K_t|\rightarrow0,\,\, P\mbox{-a.s.}
\end{equation}
\end{corollary}
\begin{proof}
It suffices to apply Theorem \ref{thm5}  and Theorem \ref{thm3}.
The uniform convergence follows from  the fact that if $a$  is a
continuous function  and $\ZH,L,U$ have continuous trajectories
then  also the solution $(X,K)$ has continuous trajectories.
\end{proof}
\medskip

Note that in the case where $L,U$ may have jumps Theorem
\ref{thm3} implies weaker then (\ref{eq5.11}) convergence. Namely,
we then have
\[
(X^n,K^n)\rightarrow(X,K),\,\,
P\mbox{-a.s.}\,\,\mbox{\rm in}\,\,\Ddd.
\]

\nsubsection{Proof of Theorem \ref{thm1}}

We follow the proof of \cite[Theorem 2.1]{fs}.

{\em Step 1.} We assume additionally that  $y^1,y^2$ and the barriers $l,u$
are step functions of the form
\[
y^j_t=Y^j_{i},\,l_t=L_i,\,u_t=U_i,\quad t\in[t_{i-1},t_{i}),
\quad i=1,\dots,n-1
\]
and $y^j_t=Y^j_{n}$, $l_t=L_{n}$, $u_t=U_{n}$,
$t\in[t_{n-1},t_n=T]$,  $j=1,2$, for some partition
$0=t_0<t_1<\dots<t_n=T$ of the interval $[0,T]$. By Remark
\ref{rem1}(b), $ k^j_t=K^j_{i}$, $t\in[t_{i-1},t_{i})$,
$i=1,\dots,n-1$,  $k^j_t=K^j_{n}$,  $t\in[t_{n-1},t_n=T]$,
$j=1,2$, where $K^1_{0}=K^2_{0}=0$ and $
K^j_{i}=\max(\min(K^j_{i-1},U_i-Y^j_{i}),L_i-Y^j_{i})$,
$i=1,\dots,n$, $j=1,2$. Clearly
\begin{equation}\label{eq2.3}
L_i-Y^j_i\leq K^j_i\leq U_i-Y^j_i,\quad i=1,\dots,n,\, j=1,2.
\end{equation}

Without loss of generality we may and will assume that
$v_p(k^1_s-k^2_s)_T> 0$. Hence there exists $i$ such that
$K^1_{i}\neq K^1_{i-1}$ or $K^2_{i}\neq K^2_{i-1}$. Later on,
without loss of generality we  will assume that for any
$i=1,\dots,n-1$,
\begin{equation}\label{eq2.4}
K^1_{i}\neq K^1_{i-1}\quad\mathrm{or}\quad
K^2_{i}\neq K^2_{i-1}
\end{equation}
(If (\ref{eq2.4}) does not hold then  we set $v_0=0$,
\[
v_k=\inf\{i>v_{k-1};K^1_{i}\neq K^1_{i-1}
\,\mathrm{or}\,K^2_{i}\neq K^2_{i-1}\}\wedge n,\quad
k=1,\dots,n,
\]
$\tilde n=\inf\{k;v_k=n\}$,  $\tilde y^j_t=Y^j_{v_k}$, $\tilde
l_t=L_{v_k}$, $\tilde u_t=U_{v_k}$, $t\in[t_{v_{k-}},t_{v_{k}})$
for $k=1,\dots,\tilde n-1$, $\tilde y^j_t=Y^j_{\tilde n}$, $\tilde
l_t=L_{\tilde n}$,  $\tilde u_t=U_{\tilde n}$, for
$t\in[t_{v_{\tilde n-1}},t_{v_{\tilde n}}=T]$, $j=1,2$. Then
(\ref{eq2.4}) holds true for the functions $\tilde y^j$, $(\tilde
x^j,\tilde k^j)=ESP(\tilde y^j,\tilde l,\tilde u)$, $j=1,2$, and
moreover, $\bar V_p(k^1-k^2)_T =\bar V_p(\tilde k^1-\tilde
k^2)_T$ and $\bar V_p(\tilde y^1-\tilde y^2)_T\leq \bar
V_p(y^1-y^2)_T$. Consequently,
\[\bar V_p(\tilde k^1-\tilde
k^2)_T\leq \bar V_p(\tilde y^1-\tilde y^2)_T\,\,\Longrightarrow \bar V_p(k^1-k^2)_T\leq\bar V_p(y^1-y^2)_T.)\]
It is clear that there exist numbers
$0=i_0<i_1<\ldots<i_m=n$ such that
\begin{equation}\label{eq2.5}
v_p(k^1-k^2)_T
=\s_{k=1}^{m}{|(K^1_{i_k}-K^1_{i_{k-1}})
-(K^2_{i_k}-K^2_{i_{k-1}})|^p}
\end{equation}
and
\begin{equation}\label{eq2.6}
(K^1_{i_k}-K^1_{i_{k-1}})-(K^2_{i_k}-K^2_{i_{k-1}})\neq 0,\quad k=1,\dots,m.
\end{equation}
Hence,  if $m\geq 2$ then for  $k=2,\dots,m$
we have
\begin{equation}\label{eq2.7}
\big((K^1_{i_{k-1}}-K^1_{i_{k-2}})-(K^2_{i_{k-1}}-K^2_{i_{k-2}})\big)
\big((K^1_{i_k}-K^1_{i_{k-1}})-(K^2_{i_k}-K^2_{i_{k-1}})\big)<0.
\end{equation}
Indeed, if (\ref{eq2.7}) is not satisfied then by (\ref{eq2.6}),
\begin{align*}
&|(K^1_{i_{k-1}}-K^1_{i_{k-2}})-(K^2_{i_{k-1}}-K^2_{i_{k-2}})|^p
+|(K^1_{i_k}-K^1_{i_{k-1}})-(K^2_{i_k}-K^2_{i_{k-1}})|^p\\
&\quad< |(K^1_{i_{k-1}}-K^1_{i_{k-2}})-(K^2_{i_{k-1}}-K^2_{i_{k-2}})
+ (K^1_{i_k}-K^1_{i_{k-1}})-(K^2_{i_k}-K^2_{i_{k-1}})|^p\\
&\quad= |(K^1_{i_k}-K^1_{i_{k-2}})-(K^2_{i_k}-K^2_{i_{k-2}})|^p,
\end{align*}
which contradicts (\ref{eq2.5}).

We will show that there exists
$0=i_0\leq r_1^\wedge\leq i_1$ (resp. $0=i_0\leq r_1^\vee\leq
i_1$)  such that if $K^1_{i_1}-K^2_{i_1}\leq 0$ (resp.
$K^1_{i_1}-K^2_{i_1}\geq 0$) then
\begin{equation}\label{eq2.8}
Y^2_{r^\wedge_1}-Y^1_{r^\wedge_1}\leq K^1_{i_1}-K^2_{i_1}\,\,
(\mbox{\rm resp.}\,K^1_{i_1}-K^2_{i_1}\leq Y^2_{r^\vee_1}-Y^1_{r^\vee_1})
\end{equation}
and for  $k=2,\dots,m$  there exist $i_{k-1}\leq
r_k^\wedge,r_k^\vee\leq i_k$ such that
\begin{equation}\label{eq2.9}
Y^2_{r^\wedge_k}-Y^1_{r^\wedge_k}\leq K^1_{i_k}-K^2_{i_k}
\leq Y^2_{r^\vee_k}-Y^1_{r^\vee_k}.
\end{equation}
Fix $k=1,\dots,m$. Set $r^j_k=\max\{i\leq i_k:
K^j_i=U_i-Y^j_i\,\mbox{\rm or} \,K^j_i=L_i-Y^j_i\}$, $j=1,2$, with
the convention that $\max\emptyset=0$.  By (\ref{eq2.4}),
$r^1_k=i_k$ or $r^2_k=i_k$. Without loss of generality we may and
will assume that $r^1_k=i_k$. Then we  have three cases:
\begin{enumerate}
\item[(a)] $K^1_{i_k}-K^2_{i_k}=(L_{i_k}-Y^1_{i_k})-(L_{r^2_{k}}-Y^2_{r^2_k})$
(or $K^1_{i_k}-K^2_{i_k}=(U_{i_k}-Y^1_{i_k})-(U_{r^2_{k}}-Y^2_{r^2_k})$),

\item[(b)] $K^1_{i_k}-K^2_{i_k}=L_{i_k}-Y^1_{i_k}$  and  $r^2_k=0$
(or $K^1_{i_k}-K^2_{i_k}=U_{i_k}-Y^1_{i_k}$  and  $r^2_k=0$),

\item[(c)] $K^1_{i_k}-K^2_{i_k}=(L_{i_k}-Y^1_{i_k})-(U_{r^2_{k}}-Y^2_{r^2_k})$
(or
$K^1_{i_k}-K^2_{i_k}=(U_{i_k}-Y^1_{i_k})-(L_{r^2_{k}}-Y^2_{r^2_k})$).
\end{enumerate}
By (\ref{eq2.3}) in all the cases
\[
K^1_{i_k}-K^2_{i_k}=L_{i_k}-Y^1_{i_k}-K^2_{i_k} \leq
L_{i_k}-Y^1_{i_k}-K^2_{i_k}+K^2_{i_k}-(L_{i_k}-Y^2_{i_k})=Y^2_{i_k}-Y^1_{i_k},
\]
which implies that we can put $r^\vee_k=i_k$. In order to  find
$r^\wedge_k$ we consider the cases (a), (b), (c) separately.

In case (a), if $r^2_k=i_k$  then
\[
K^1_{i_k}-K^2_{i_k}=L_{i_k}-Y^1_{i_k}-(L_{i_k}-Y^2_{i_k})=Y^2_{i_k}-Y^1_{i_k}
\]
and we put $r^\wedge_k=i_k$. If  $r^2_k<i_k$  then we set
$r^\star=\max\{i<i_k:K^1_i=U_i-Y^1_i\}\vee r^2_k$. Observe that
$K^2_{r^\star}=K^2_{r^\star+1}=\dots=K^2_{i_k}$. Since  for
$r^\star<v\leq i_k$, $K^1_v=\max(K_{v-1},L_v-Y^1_v)$, it follows
by (\ref{eq2.4}) that
\begin{equation}\label{eq2.10}
K^1_{r^\star}-K^2_{r^\star}<K^1_{r^\star+1}
-K^2_{r^\star+1}<\dots<K^1_{i_k}-K^2_{i_k}.
\end{equation}
From this it also follows that
$Y^2_{r^\star}-Y^1_{r^\star}<K^1_{i_k}-K^2_{i_k}$. Indeed, if
$r^\star>r^2_k$ (resp. $r^\star=r^2_k$ ) then
$K^1_{r^\star}=U_{r^\star}-Y^1_{r^\star}$  (resp.
$K^2_{r^\star}=L_{r^\star}-Y^2_{r^\star}$ )  and by (\ref{eq2.3}),
\[
K^1_{i_k}-K^2_{i_k}>K^1_{r^\star}-K^2_{r^\star}
\geq U_{r^\star}-Y^1_{r^\star}-U_{r^\star}+Y^2_{r^\star}
=Y^2_{r^\star}-Y^1_{r^\star},
\]
(resp.
$\displaystyle{K^1_{i_k}-K^2_{i_k}>K^1_{r^\star}-K^2_{r^\star}\geq
L_{r^\star}-Y^1_{r^\star}-L_{r^\star}+Y^2_{r^\star}=Y^2_{r^\star}-Y^1_{r^\star}}$).
What is left is  to put  $r^\wedge_k=r^\star$ and  show that
$i_{k-1}\leq r^\star$. This is obvious if $k=1$, so assume that
$k\geq 2$ and $i_{k-1}>r^\star$. By (\ref{eq2.10}),
$(K^1_{i_k}-K^1_{i_{k-1}})-(K^2_{i_k}-K^2_{i_{k-1}})>0$. From this
(\ref{eq2.7}) it follows that
$(K^1_{i_{k-1}}-K^1_{i_{k-2}})-(K^2_{i_{k-1}}-K^2_{i_{k-2}})<0$,
which together with (\ref{eq2.10}) implies that $i_{k-2}<r^\star$.
Using once again (\ref{eq2.10}) we see that
$(K^1_{r^\star}-K^1_{i_{k-1}})-(K^2_{r^\star}-K^2_{i_{k-1}})<0$.
Consequently,
\begin{align*}
0&>(K^1_{i_{k-1}}-K^1_{i_{k-2}})-(K^2_{i_{k-1}}-K^2_{i_{k-2}})\\
&>(K^1_{i_{k-1}}-K^1_{i_{k-2}})-(K^2_{i_{k-1}}-K^2_{i_{k-2}})
+(K^1_{r^\star}-K^1_{i_{k-1}})-(K^2_{r^\star}-K^2_{i_{k-1}})\\
&=(K^1_{r^\star}-K^1_{i_{k-2}})-(K^2_{r^\star}-K^2_{i_{k-2}})
\end{align*}
and
\begin{equation}\label{eq2.11}
|(K^1_{i_{k-1}}-K^1_{i_{k-2}})-(K^2_{i_{k-1}}-K^2_{i_{k-2}})|^p
<|(K^1_{r^\star}-K^1_{i_{k-2}})-(K^2_{r^\star}-K^2_{i_{k-2}})|^p.
\end{equation}
Similarly,
\begin{align*}
0&<(K^1_{i_k}-K^1_{i_{k-1}})-(K^2_{i_k}-K^2_{i_{k-1}})\\
&<(K^1_{i_k}-K^1_{i_{k-1}})-(K^2_{i_k}-K^2_{i_{k-1}})
-(K^1_{r^\star}-K^1_{i_{k-1}})+(K^2_{r^\star}-K^2_{i_{k-1}})\\
&=(K^1_{i_k}-K^1_{r^\star})-(K^2_{i_k}-K^2_{r^\star}),
\end{align*}
which  implies that
\begin{equation}\label{eq2.12}
|(K^1_{i_k}-K^1_{i_{k-1}})-(K^2_{i_k}-K^2_{i_{k-1}})|^p
<|(K^1_{i_k}-K^1_{r^\star})-(K^2_{i_k}-K^2_{r^\star})|^p.
\end{equation}
Combining (\ref{eq2.11}) with (\ref{eq2.12}) we obtain
\begin{align*}
&|(K^1_{i_{k-1}}-K^1_{i_{k-2}})-(K^2_{i_{k-1}}-K^2_{i_{k-2}})|^p
+|(K^1_{i_k}-K^1_{i_{k-1}})-(K^2_{i_k}-K^2_{i_{k-1}})|^p\\
&\quad<|(K^1_{r^\star}-K^1_{i_{k-2}})-(K^2_{r^\star}-K^2_{i_{k-2}})|^p
+|(K^1_{i_k}-K^1_{r^\star})-(K^2_{i_k}-K^2_{r^\star})|^p,
\end{align*}
which contradicts  (\ref{eq2.5}) and completes the proof of the
fact that $i_{k-1}\leq r^\star$. Consequently, in case (a) we put
$r^\wedge_k=r^\star$.

In case (b) (resp. (c)) we set  $r^\star=\max\{i< i_k:
K^1_i=U_i-Y^1_i\}$ (resp.  $r^\star=\max\{i< i_k:
K^2_i=L_i-Y^2_i\,\mathrm{or}\, K^1_i=U_i-Y^1_i\}$). For
$r^\star<v\leq i_k$ we have $K^1_v=\max(K^1_{v-1},L_v-Y^1_v)$ and
$K^2_v=\min(K^2_{v-1},U_v-Y^2_v)$. As in case (a) we conclude from
this and (\ref{eq2.4}) that that \[ K^1_{r^\star}-K^2_{r^\star}<
K^1_{r^\star+1}- K^2_{r^\star+1}<\dots< K^1_{i_{k}}-K^2_{i_{k}}.\]
By the argument used in case (a) we also show that $i_{k-1}\leq
r^\star$. Moreover, if $r^\star>0$ and
$K^2_{r^\star}=L_{r^\star}-Y^2_{r^\star}$  (resp.
$K^1_{r^\star}=U_{r^\star}-Y^1_{r^\star}$) then by (\ref{eq2.3}),
\[
K^1_{i_k}-K^2_{i_k}> K^1_{r^\star}-(L_{r^\star}-Y^2_{r^\star})
\geq(L_{r^\star}-Y^1_{r^\star})
-(L_{r^\star}-Y^2_{r^\star})= Y^2_{r^\star}-Y^1_{r*}
\]
(resp. $\displaystyle{K^1_{i_k}-K^2_{i_k}>
(U_{r^\star}-Y^1_{r^\star})-K^2_{r^\star}
\geq(U_{r^\star}-Y^1_{r^\star})-(U_{r^\star}-Y^2_{r^\star})=
Y^2_{r^\star}-Y^1_{r^\star}}$). Therefore we put
$r^\wedge_k=r^\star$. Since $k=1$ if $r^\star=0$, the proof of
(\ref{eq2.8}) and  (\ref{eq2.9}) is complete.

Now observe that by (\ref{eq2.9}), if
$K^1_{i_k}-K^2_{i_k}>K^1_{i_{k-1}}-K^2_{i_{k-1}}$ for some
$k=2,\dots,m$ then
\[
0<(K^1_{i_k}-K^2_{i_k})-(K^1_{i_{k-1}}-K^2_{i_{k-1}})
\leq(Y^2_{r^\vee_k}-Y^1_{r^\vee_k})
-(Y^2_{r^\wedge_{k-1}}-Y^1_{r^\wedge_{k-1}}),
\]
which implies that
\begin{equation}\label{eq2.13}
|(K^1_{i_k}-K^2_{i_k})-(K^1_{i_{k-1}}-K^2_{i_{k-1}})|^p
\leq|(Y^1_{r^\vee_k}-Y^2_{r^\vee_k})
-(Y^1_{r^\wedge_{k-1}}-Y^2_{r^\wedge_{k-1}})|^p.
\end{equation}
Similarly, if
 $K^1_{i_k}-K^2_{i_k}<K^1_{i_{k-1}}-K^2_{i_{k-1}}$ then
\begin{equation}\label{eq2.14}
|(K^1_{i_k}-K^2_{i_k})-(K^1_{i_{k-1}}-K^2_{i_{k-1}})|^p\leq
|(Y^1_{r^\wedge_k}-Y^2_{r^\wedge_{k}})-
(Y^1_{r^\vee_{k-1}}-Y^2_{r^\vee_{k-1}})|^p.
\end{equation}
In case $k=1$, if $K^1_{i_1}-K^2_{i_1}>0$ then
\begin{equation}\label{eq2.15}
|(K^1_{i_1}-K^2_{i_1})-(K^1_{i_0}-K^2_{i_0})|^p
=|K^1_{i_1}-K^2_{i_1}|^p\leq|Y^1_{r^\vee_1}- Y^2_{r^\vee_1}|^p
\end{equation}
and if $K^1_{i_1}-K^2_{i_1}<0$ then
\begin{equation}\label{eq2.16}
|(K^1_{i_1}-K^2_{i_1})-(K^1_{i_0}-K^2_{i_0})|^p
\leq|Y^1_{r^\wedge_1}- Y^2_{r^\wedge_1}|^p.
\end{equation}
Putting together (\ref{eq2.13})--(\ref{eq2.16}) we conclude
that
\[
\s_{k=1}^{m}|(K^1_{i_k}-K^2_{i_k})-(K^1_{i_{k-1}}-K^2_{i_{k-1}})|^p
\leq |Y^1_{\tilde{r}_1}- Y^2_{\tilde{r}_1}|^p
+\s_{k=2}^{m}|(Y^1_{\tilde{r}_k}-Y^2_{\tilde{r}_{k}})-
(Y^1_{\tilde{r}_{k-1}}-Y^2_{\tilde{r}_{k-1}})|^p,
\]
where $\tilde{r}_k=r^\wedge_k$ or $\tilde{r}_k=r^\vee_k$  and
$i_{k-1}\leq \tilde{r}_k\leq i_k$, $k=1,\dots,m$. Hence
\begin{align*}
\bar V_p(k^1-k^2)_T&=V_p(k^1-k^2)_T\\
&=\Big(\s_{k=1}^{m}|(K^1_{i_k}-K^2_{i_k})
-(K^1_{i_{k-1}}-K^2_{i_{k-1}})|^p\Big)^{1/p}\\
&\leq \Big(|Y^1_{\tilde{r}_1}- Y^2_{\tilde{r}_1}|^p
+\s_{k=2}^{m}|(Y^1_{\tilde{r}_k}-Y^2_{\tilde{r}_{k}})
-(Y^1_{\tilde{r}_{k-1}}-Y^2_{\tilde{r}_{k-1}})|^p\Big)^{1/p}\\
&\leq |y^1_0- y^2_0|
+\Big(\s_{k=1}^{m}|(y^1_{t_{\tilde{r}_k}}-y^2_{t_{\tilde{r}_{k}}})-
(y^1_{t_{\tilde{r}_{k-1}}}-y^2_{t_{\tilde{r}_{k-1}}})|^p\Big)^{1/p}
\end{align*}
for some partition
$0=t_{\tilde{r}_0}<t_{\tilde{r}_1}<\dots<t_{\tilde{r}_m}\leq T$,
which proves the theorem in the case of step functions $y^1,y^2$
and step barriers $l,u$.

{\em Step 2.} The general case.

Let $\{y^{1,n}\}$, $\{y^{2,n}\}$,
$\{l^n\}$ and $\{u^n\}$ be sequences of discretizations of $y^1$,
$y^2$, $l$ and $u$, respectively, i.e. $y^{1,n}_t=y^1_{k/n}$,
$y^{2,n}_t=y^{2}_{k/n}$,$l^{n}_t=l_{k/n}$, $u^{n}_t=u_{k/n}$
$t\in[k/n,(k+1)/n)$, $k\in\No$. By \cite[Chapter VI, Proposition
2.2]{js},  $(y^{1,n},y^{2,n},l^n,u^n)\lra(y^1,y^2,l,u)$ in
$\DDDD$. Let $(x^{j,n},k^{j,n})=ESP(y^{j,n},l^n,u^n)$, $n\in\N$,
$j=1,2$. By (\ref{eq2.2}),
$(k^{1,n},k^{2,n},y^{1,n},y^{2,n})\lra(k^1,k^2,y^1,y^2)$ in
$\DDDD$, which implies that
\begin{equation}\label{eq2.17}
k^{1,n}-k^{2n}\lra k^1-k^2\quad\mbox{\rm in}\,\,\D.
\end{equation}
By Step 1, for $n\in\N$ and $T\in\Rp$ we have $\bar
V_p(k^{1,n}-k^{2,n})_T\leq \bar V_p (y^{1,n}-y^{2,n})_T$. Clearly,
$\bar V_p(y^{1,n}-y^{2,n})_T\leq \bar V_p(y^1-y^2)_T$, $n\in\N$,
$T\in\Rp$. From this and (\ref{eq2.17}) it follows that for every
$T\in\Rp$ such that $\Delta k^1_T=\Delta k^2_T=\Delta y^1_T=\Delta
y^2_T=0$,
\begin{align*}
\bar V_p(k^1-k^2)_T \leq\liminf_{n\to\infty} \bar
V_p(k^{1,n}-k^{2,n})_T \leq\sup_n \bar V_p(y^{1,n}-y^{2,n})_T\leq
\bar V_p(y^1-y^2)_T.
\end{align*}
To obtain the desired result for arbitrary $T\in\Rp$ we use right
continuity of $\bar V_p(k^{1}-k^{2})$ and $\bar V_p(y^{1}-y^{2})$.

\mbox{}\\
{\bf Acknowledgements}\\[1mm]
{Research supported by Polish NCN grant no.  2012/07/B/ST1/03508).}


\begin{thebibliography}{10}
\bibitem{abb}
V. Acary,  O. Boneffon and B. Brogliato, Nonsmooth Modeling and
Simulation for Switched Circuits, Lecture Notes in Electrical
Engineering, vol. 69 Springer, Dordrecht (2011). ISBN
978-90-481-9680-7.

\bibitem{aht}
S. Adly,  T. Haddad and L. Thibault, Convex sweeping process  in
the framework of measure differential inclusions and evolution
variational inequalities, Math. Program. Ser. B (2014) DOI
10.1007/s10107-014-0754-4.
\bibitem{as}
S. Asmussen, Queueing simulation in heavy traffic, {Math.
Oper. Res.} {17}  (1992) 84--111.

\bibitem{be}
H. Benabdellah, Existence of solutions to the noncovex sweeping
process, J. Differential Equations 164 (2000) 286-295.

\bibitem{bv}
F. Bernicot, J. Venel,  Stochastic perturbation of sweeping proces
and a convergence result for an associated numerical scheme, J.
Differential Equations 251 (2011) 1195-1224.

\bibitem{bkr}
K. Burdzy, W. Kang and K.  Ramanan,  The Skorokhod problem in a
time-dependent interval, Stoch. Process. Appl. {2009},  {
119}, 428--452.

\bibitem{co1}
C. Castaing, Version al\'eatoire du probl\`eme de rafle par  un
convexe variable,  C. R.  Acad. Sci. Paris Ser. A 277(1973)
1057-1059.

\bibitem{co2}
C. Castaing, Equations diff\'erentielles. Rafle par un convexe
al\'eatoire \`a variation continue\`a droite,  C. R.  Acad. Sci.
Paris Ser. A  282 (1976) 515-518.

\bibitem{cdhv}
C. Castaing, T. X. D\'uc Ha and M. Valadier, Evolution equations
governed  by the sweeping  process, Set-Valued Anal. 1 (1993)
109-139.

\bibitem{cm}
C. Castaing, M.D.P. Monteiro Marques, BV periodic solutions  of an
evolution problem associated with  continuous moving convex sets,
Set-Valued Anal.  3 (4) (1995) 381-399.

\bibitem{cr}
I. Ciotir, A. R\u{a}\c{s}canu, Viability for differential
equations  driven by fractional Brownian motion, J. Differential
Equations 247 (2009) 1505-1528.

\bibitem{cg}
G. Colombo, V.V. Goncharov, The sweeping processes without
convexity, Set-Valued Anal. 7 (1999) 357-374.

\bibitem{cmm}
G. Colombo, M.D.P. Monteiro Marques,  Sweeping by a continuous
prox-regular set, J. Differential Equations 187 (1) (2003) 46-62.

\bibitem{cs}
F. Coquet, L. S{\l}omi{\'n}ski, {On the convergence of {D}irichlet
processes},  Bernoulli {5} (1999) 615--639.

\bibitem{cms}
F. Coquet, J. M\'emin  and  L. S{\l}omi{\'n}ski, {On
non-continuous {D}irichlet processes}, Journal of Theoretical
Probability {16} (2003), 197--216.
\bibitem{dkt}
P. Dr\'abek, P. Krej\u{c}i and P. Taka\u{c}, Nonlinear Differential Equations, Research Notes in Mathematics vol. 404, Chapman $\&$ Hall, Lonndon 1999.
\bibitem{du} R. M. Dudley, {Picard iteration and p-variation: The work of Lyons (1994)}, Mini-proceedings: Workshop on Product Integrals and Pathwise Integration, MaPhySto 1999.
\bibitem{dn} R. M. Dudley, R. Norvai\v{s}a, {An Introduction to p-variation and Young Integrals}, Lecture Notes No. 1, Aarhus University, (1999).
\bibitem{dn1} R. M. Dudley, R. Norvai\v{s}a, {Concrete Functional Calculus}, Springer Science+Business Media 2011.
    \bibitem{di}
P. Dupuis, H. Ishii,  {On lipschitz continuity of the solution mapping
to the Skorokhod problem, with
applications}, Stochast. Stochast. Rep. {35} (1991) 31--62.
\bibitem{dr}
P. Dupuis, K. Ramanan, A multiclass feedback queueing network with
a regular Skorokhod problem,  { Queueing Systems} { 36}
(2000) 327--349.
\bibitem{ek}
{S. Ethier, T. Kurtz}, { Markov  Processes. Charakterization and Convergence,} John  Wiley $\&$ Sons, New York, 1986.
\bibitem{fs}
A. Falkowski, L. S{\l}omi{\'n}ski,  { SDEs with constraints driven by processes with bounded $p$-variation},
(2014), submitted.
\bibitem{fr1}
M. Ferrante, C. Rovira, {Stochastic differential equations with non-negativity constraints driven by  fractional Brownian motion}, J. Evol. Equ. {13} (2013) 617--632.
\bibitem{js}
{J. Jacod, A. Shiryaev}, { Limit Theorems for Stochastic Processes,} Springer Verlag, Berlin  1987.

\bibitem{jm}
N. C. Jain, D. Monrad, {Gaussian measures in $B_p$}, Ann. Probab. {11} (1983) 46--57.

\bibitem{jmp}
A. Jakubowski, J. M{\'e}min, G. Pag{\`e}s, {Convergence en loi des suites
d'int{\'e}grales stochastiques sur l'espace ${D^1}$ de {S}korokhod}, Probab.
  Theory Related Fields {81} (1989) 111--137.

\bibitem{KS}
P. Kr\'ee, C. Soize, {Mathematics and Random Phenomena},
Dordrecht: Reidel 1986.

\bibitem{k1}
K. Kubilius, {The existence and uniqueness of the solution of an
integral equation driven by a $p$-semimartingale of special type},
Stochastic Process. Appl. {98} (2002) 289--315.

\bibitem{k2}
K. Kubilius, {On weak and strong solutions of an integral equation
driven by a continuous $p$-semimartingale}, Lithuanian Math. J.
{43} (2003) 34--50.

\bibitem{k3}
K. Kubilius, {On weak solutions of an integral equation driven by
a $p$-semimartingale of special type}, Acta Appl. Math. {78}
(2003)
  223--242.

\bibitem{ls}
P. L. Lions, A. S. Sznitman, {Stochastic Differential  Equations
with Reflecting Boundary Conditions}, Comm. Pure App. Math.
{XXXVII} (1983) 511--537.

\bibitem{ly}
T. J. Lyons, {Differential equations driven by rough signals (I):
An extensions of an inequality of L. C. Young}, Math. Res. Lett.
{1} (1994) 451--464.
\bibitem{mm}
M. D. P. Monteiro Marques, Differential Inclusions in Nonsmooth
Mechanical Problems, Shocs and Dry Friction, Birkhauser, Basel
(1993).

\bibitem{mo1}
J. J. Moreau, Rafle par un convexe  variable I, S\'em. Anal.
Convexe Montpelier (1971), Expos\'e 15.

\bibitem{mo2}
J. J. Moreau, Rafle par un convexe  variable II, S\'em. Anal.
Convexe Montpelier (1972), Expos\'e 3.

\bibitem{mo3}
J. J. Moreau, Evolution problem associated with a moving convex
set in a Hilbert space, J. Differential Equations 26 (1977),
347--374.

\bibitem{mmv}
J. M\'emin, Y. Mishura, E. Valkeila, {Inequalities for
the moments of Wiener integrals with respect to fractional
Brownian motions},  Statist. Probab. Lett. {44} (2001)
197--206.
\bibitem{nr}
D. Nualart, A. R\u{a}\c{s}canu, { Differential equations
driven by  fractional  Brownian motion}, Collect. Math.
{53} (2002) 55--81.
\bibitem{ro}
A. Rozkosz,  On a decomposition of symmetric diffusions with reflecting boundary conditions,
Stochastic Process. Appl. {103} (2003) 101--122.
\bibitem{ru}
A. A. Ruzmaikina, {Stieltjes integrals of {H}{\"o}lder continuous
  functions with applications to fractional {B}rownian motion}, J. Statist.
  Phys. {100} (2000) 1049--1069.
\bibitem{Sa}
Y. Saisho, { Stochastic differential equations for multi--dimensional
domain
with reflecting boundary},
{  Probab. Th. Rel. Fields}  { 74} (1987)
455--477.
 \bibitem{sa}
M. A. Shashiashvili, {On the  Variation of the  Difference of  Singular
Components in the Skorokhod Problem  and on  Stochastic  Differential
Systems in a Half-Space}, Stochastics {24} (1988) 151--169.
\bibitem{ss}
L. A. Shepp, A. N. Shiryaev, A new look  at pricing of the
"Russian option", { Theory Probab. Appl.} { 39} (1994)
103--119.
\bibitem{sk}
A. V. Skorokhod, {Stochastc equations for diffusion processes in a bounded
region 1,2}, Theory Probab. Appl. {6} (1961) 264--274, {7}
(1962) 3--23.


\bibitem{s4}
L. S{\l}omi{\'n}ski, {On existence,  uniqueness and stability of solutions
of multidimensional SDE's with reflecting boundary conditins},
Ann. Inst. H. Poincar\'e {29} (1993) 163--198.
\bibitem{s3}
L. S{\l}omi{\'n}ski, {Euler's approximations of solutions of
 SDEs with reflecting boundary}, Stochastic Process. Appl. {94} (2001) 317--337.

\bibitem{sw}
L. S{\l}omi{\'n}ski, T. Wojciechowski, {Stochastic differential
equations with time-dependent reflecting barriers}, Stochastics
{85} (2013) 27--47.

\bibitem{sz2}
L. S{\l}omi{\'n}ski, B. Ziemkiewicz, {On weak approximations of
integrals with respect to  fractional {B}rownian motion}, Statist.
Probab. Lett. {79} (2009) 543--552.

\bibitem{ta}
T. Tanaka, {Stochastic Differential Equations with Reflecting
Boundary Condition in Convex Regions}, Hiroshima Math. J. {9} (1979) 163--177.
\bibitem{th}
L. Thibault, Sweeping process with regular and nonregular sets, J.
Differential Equations 193(1) (2003) 1--26.
\end{thebibliography}
\end{document}